\documentclass[12pt]{amsart}
\usepackage{geometry}
\usepackage{amssymb,amsmath,amsthm}
\usepackage[english]{babel}

\usepackage{nicematrix}

\usepackage{tikz}

\newcommand{\length}{\operatorname{\lambda}}

\newcommand{\RR}{\mathbb {R}}

\newcommand{\mf}[1]{\mathfrak #1}

\newcommand{\ehk}{\operatorname {e_{HK}}}

\newcommand{\eh}{\operatorname {e}}

\newcommand{\bZ}{{\mathbf Z}}
\newcommand{\bN}{{\mathbf N}}
\newcommand{\bT}{{\mathbf T}}
\newcommand{\bM}{{\mathbf M}}

\newcommand{\En}[1]{{\mathcal A}_{#1}} 

\newcommand{\Ehr}[1]{\operatorname{Ehr}_{#1}}

\DeclareMathOperator{\Mat}{\mathbf{M}}

\DeclareMathOperator{\swap}{swap}
\DeclareMathOperator{\sign}{sign}

\DeclareMathOperator{\chr}{char}
\DeclareMathOperator{\modt}{Mod_T}

\newcommand{\cotimes}[1]{\mathbin{\widehat{\otimes_{#1}}}}

\theoremstyle{plain}

\newtheorem{conjecture}{Conjecture}
\newtheorem{theorem}{Theorem}[section]
\newtheorem{lemma}[theorem]{Lemma}
\newtheorem{proposition}[theorem]{Proposition}
\newtheorem{corollary}[theorem]{Corollary}
\newtheorem{claim}{Claim}[theorem]
\newtheorem{procedure}{Procedure}
\newtheorem*{theorem*}{Theorem}
\theoremstyle{definition}
\newtheorem{definition}[theorem]{Definition}
\newtheorem*{definition*}{Definition}

\newtheorem{example}[theorem]{Example}
\newtheorem{remark}[theorem]{Remark}

\begin{document}

\title[Hilbert--Kunz multiplicity and Ehrhart theory]{Hilbert--Kunz multiplicity of quadrics via Ehrhart theory}

\author{Igor Pak}
\address{Department of Mathematics, UCLA, Los Angeles, CA, 90095, USA} 

\author{Boris Shapiro}
\address{Department of Mathematics, Stockholm University, SE-106 91 Stockholm, Sweden
\quad and \quad 
Department of
Mathematics with Computer Science, Guangdong Technion – Israel Institute of Technology, Shantou, China
}

\author{Ilya Smirnov}
\address{BCAM -- Basque Center for Applied Mathematics, Mazarredo 14, 48009 Bilbao, Spain \quad and \quad
IKERBASQUE, Basque Foundation for Science, Plaza Euskadi 5, 48009 Bilbao, Spain}

\author{Ken-ichi Yoshida}
\address{Department of Mathematics, College of Humanities and Sciences, Nihon University, 3-25-40 Sakurajosui, Setagaya-ku, Tokyo 156-8550, Japan}


\begin{abstract} 
We show that the Hilbert--Kunz multiplicity $h_{p, d}$ of the $d$-dimensional non-degenerate quadric hypersurface of characteristic $p > 2$  is a rational function of $p$ composed from the Ehrhart polynomials of integer polytopes.
In consequence, we recover and explain the Gessel--Monsky formula for the limit $\lim_{p \to \infty} h_{p, d}$, and prove that $h_{p, d}$ is a decreasing function of both arguments. 
 \end{abstract}

\maketitle

\section{Introduction}
Hilbert--Kunz multiplicity is a multiplicity theory native to positive characteristic, one of the numerical invariants defined via the iterates of the Frobenius endomorphism.

\begin{definition*}[Monsky]
Let $(R, \mf m)$ be a commutative Noetherian local ring of positive characteristic $p > 0$. 
Then the \emph{Hilbert--Kunz multiplicity} of $R$ is defined as the limit
\[
\ehk(R) := \lim_{e \to \infty} \frac{\length (R/\mf m^{[p^e]})}{p^{e \dim R}},
\]
where $\mf m^{[p^e]}$ is the ideal generated by all $p^e$th powers of elements in $\mf m$. 
\end{definition*}

The ideal $\mf m^{[p^e]}$ is called the \emph{Frobenius power} of $\mf m$; this terminology indicates the analogy to Samuel's definition of the classical Hilbert--Samuel multiplicity. Hilbert--Kunz multiplicity measures the failure of flatness of the Frobenius endomorphism. The theory originates in the work of Kunz \cite{Kunz1, Kunz2}: \cite{Kunz1} shows that flatness of Frobenius characterizes regular rings, while in \cite{Kunz2} he studies the sequence whose limit is the Hilbert--Kunz multiplicity to provide a quantitative measure of this failure. The invariant was defined by Monsky \cite{Monsky}, whose motivation came from the Iwasawa theory. 

From its very origins, the development of the Hilbert--Kunz theory was motivated by its use as a tool of understanding singularities and its similarity with the Hilbert--Samuel theory. For example, the celebrated theorem of Nagata (\cite{Nagata}) asserts that a local ring $R$ is regular if and only if it is unmixed\footnote{The unmixed condition means that the completion $\hat{R}$ has no lower-dimensional components -- lower-dimensional components do not contribute to the limit.}
and its multiplicity is $1$. 
In parallel, Watanabe and Yoshida show in \cite{WatanabeYoshida} that a local ring $R$ of characteristic $p> 0$ is regular if and only if it is unmixed and $\ehk(R) = 1$. 

However, there is a major difference between these two theories of multiplicity: 
Hilbert--Kunz multiplicity need not be an integer; in fact, it may be even irrational (\cite{Brenner}). Overall, little is known about the set of values of Hilbert--Kunz multiplicity. A very natural question was raised by Blickle and Enescu in \cite{BlickleEnescu}: what is 
\[\inf \left \{R \mid \ehk(R) > 1, \chr R = p, \dim R = d \right \}?\]
In other words, they asked for the best lower bound on the Hilbert--Kunz multiplicity of an unmixed singular ring. 
Blickle and Enescu provided the lower bound $\ehk(R) \geq 1 + \frac{1}{d!p^d}$, 
showing that $1$ is not a limit point in the set of Hilbert--Kunz multiplicities. Another lower bound was given in \cite{CDHZ}.

In a different direction, Watanabe and Yoshida \cite{WatanabeYoshida3d} conjectured what singular rings minimize Hilbert--Kunz multiplicity. 

\begin{conjecture}\label{conj WY 1}
Let $p \geq 3$ be a prime number and define the simple $(A_1)$-singularity\footnote{As discussed in the unpublished notes of the last-named author \cite{Yoshida} the equation of $A_{2,d}$ should be different, see also \cite{JNSWY, CastilloRey}.} by the equation
\[
A_{p,d} := \mathbb{F}_{p} [[x_0, \ldots, x_d]]/(x_0^2 + \cdots + x_d^2).
\]
Let $(R,\mf m)$ be a local ring of dimension $d \geq 1$ and characteristic $p > 0$, and set $k = R/\mf m$. 
Then
\begin{enumerate}
\item if $\ehk(R) \neq 1$, then $\ehk(R) \geq \ehk(A_{p,d})$,  
\item if $R$ is unmixed and $k$ is algebraically closed, then 
$\ehk(R) = \ehk(A_{p,d})$ if and only if
\[
\widehat{R} \cong k[[x_0, \ldots, x_d]]/(x_0^2 + \cdots + x_d^2)
= A_{p,d} \cotimes{\mathbb{F}_p} k. 
\]
\end{enumerate}
\end{conjecture}

Conjecture~\ref{conj WY 1} is now known in dimensions at most $8$ (\cite{WatanabeYoshida3d, AberbachEnescu1, AberbachEnescu2, CoxSteibAberbach, CastilloRey}) and for complete 
intersections (\cite{EnescuShimomoto, CastilloRey}). 

An important subtlety is that Conjecture~\ref{conj WY 1} does not directly provide a numerical lower bound, as a closed formula for $\ehk(A_{p,d})$ is not known. In her thesis \cite{Han}, Han developed an algorithm computing the Hilbert--Kunz multiplicity of diagonal hypersurfaces, which was published in an improved form in \cite{HanMonsky}. 
The algorithm allows to compute $\ehk(A_{p,d})$ for specific values of $p$ and $d$, but a closed formula was only presented for $d \leq 4$. However, Gessel and Monsky \cite{GesselMonsky} used the algorithm to find the asymptotic behavior of $\ehk(A_{p,d})$ as $p$ grows large:
\[
\lim_{p \to \infty} \ehk(A_{p,d}) = 1 + \frac{\En{d}}{d!}
\]
where $\En{d}$ are the Euler (zigzag) numbers given by the Taylor--Maclaurin series
\[
\sec x + \tan x = \sum \frac{\En{d}}{d!} x^d.
\] 
This result motivated the following 
characteristic-free bound.

\begin{conjecture}[Watanabe--Yoshida, \cite{WatanabeYoshida3d}]\label{conj WY 2}
For any dimension $d \geq 1$ and characteristic $p \geq 3$, one has 
$\ehk(A_{p,d}) \geq 1 + \En{d}/d!$. 
\end{conjecture}

Recently, in \cite{Trivedi}, Trivedi showed that this conjecture is true for any $d$ and $p > d-1$.
Even more recently, Meng announced a complete\footnote{Meng assumes that $p \geq 3$, but the conjecture holds for $p = 2$ by a direct computation in \cite{CastilloRey} if $A_{2,d}$ is appropriately defined.} solution in \cite{Meng}.
Combined with the result of Enescu and Shimomoto \cite{EnescuShimomoto}
there is now a good characteristic-free bound on Hilbert--Kunz multiplicity 
of complete intersections, but for general local rings we only have
a much weaker characteristic-free bound observed in \cite{AberbachEnescu1}: 
\[\ehk(R) \geq 1 + \frac{1}{d(d!(d-1) + 1)^d}.\]
Note that it is known that $\En{d}/d! \sim 2 (2/\pi)^{d+1}$.

It is believed that the bound in Conjecture~\ref{conj WY 2} 
is not sharp. In fact, the last named author conjectured in the unpublished note \cite{Yoshida}
that $\ehk(A_{p,d})$ is a strictly decreasing sequence in $p$ for a fixed $d$.
An affirmative answer for  $p$ sufficiently large (depending on $d$) 
was shown by Trivedi in \cite{Trivedi} as a byproduct of her approach to Conjecture~\ref{conj WY 2}. Very recently, this result was refined to full generality  by Meng in \cite{MengMonotonicity}.

\subsection*{Results}
This note provides an approach to Conjecture~\ref{conj WY 2} that is drastically different from works of Trivedi and Meng. Trivedi's work \cite{Trivedi} is quite intricate, it builds on Achinger's computation of Frobenius pushforwards of vector bundles on quadrics in order to study the \emph{Hilbert--Kunz density functions} (a theory developed in \cite{TrivediDensity}) and then treats the densities by analytic tools. Meng's improvements in \cite{Meng, MengMonotonicity} result from even more sophisticated analytic tools. 
The approach of this paper is more elementary, directly based on the algorithm of Han and Monsky \cite{HanMonsky}. 
For quadrics the algorithm can be interpreted using linear algebra (\cite{Yoshida}): for the families of $(2n + 1) \times (2n+1)$ square matrices
\[
\bT_n = 
\begin{bNiceMatrix}[nullify-dots]
2 	    &	 \Cdots   &   2	 & 1   	  & 0  	&  \Cdots & 0  \\
\Vdots &	 \Iddots  & \Iddots& \Vdots & \Ddots&  \Ddots & \Vdots  \\
2 	    &	 \Iddots  &           &    	  &         &  \Ddots & 0  \\
1	    &	\Cdots  	& 		 & 1   	  &\Cdots&  		   & 1  \\
0	    &	  	       & 		 &\Vdots  &   	&  		   & 2  \\
\Vdots &	 \Ddots  	 & \Ddots&    	  &\Iddots&\Iddots   &\Vdots  \\
0	    &	\Cdots  	 & 0		 &    1	  & 2  	&\Cdots    & 2  \\
\end{bNiceMatrix} \text { and }
\bN_n = 
\begin{bNiceMatrix}[nullify-dots]
0 	    &	 \Cdots   &   0	 & 1   	  & 0  	&  \Cdots & 0  \\
\Vdots &	 \Iddots  & \Iddots& \Vdots & \Ddots&  \Ddots & \Vdots  \\
0 	    &	 \Iddots  &           &    	  &         &  \Ddots & 0  \\
1	    &	\Cdots  	& 		 & 1   	  &\Cdots&  		   & 1  \\
0	    &	  	       & 		 &\Vdots  &   	&  		   & 0  \\
\Vdots &	 \Ddots  	 & \Ddots&    	  &\Iddots&\Iddots   &\Vdots  \\
0	    &	\Cdots  	 & 0		 &    1	  & 0  	&\Cdots    & 0  \\
\end{bNiceMatrix}.
\]
one computes (Corollary~\ref{cor Matrix Form}) that
\[
\ehk(A_{p,d}) = 1 + \frac{[\bT_{\lfloor p/2 \rfloor}^{d+1}]_{1, 1} - p^d}{p^d - [\bN_{\lfloor p/2 \rfloor}^{d+1}]_{1, 1}}, 
\]
where $[\bM]_{1,1}$ denotes the $(1,1)$–entry of a matrix $\bM$. 
Using the special shape of the matrices $\bT_n$ and $\bN_n$, we prove the following.

\begin{theorem*}[{Corollaries~\ref{cor main Ehrhart},~\ref{cor: rational function},~\ref{cor convergence rate},~\ref{cor: monotone function}}]
Let $\Ehr{F_d}(n)$ and $\Ehr{E_d}(n)$ be the Ehrhart polynomials of the $d$-dimensional \emph{Fibonacci} and \emph{extended Fibonacci} polytopes (Definition~\ref{def Fibonacci}). Then for all $p \geq 3$, the Hilbert--Kunz multiplicity of $A_{p,d}$ is given by 
\[
\ehk(A_{p,d}) = 1 + \frac{2^d \Ehr{F_d} \left (\frac{p-3}{2} \right)}{p^d - \Ehr{E_{d-2}} \left(\frac{p-1}{2} \right)}.
\]
As a consequence, for any $d \geq 1$ we have
\begin{enumerate}
\item there exist polynomials $f, g \in \mathbb{Q}[x]$ of degree $\lfloor d/2 \rfloor - 1$ such that 
$\ehk(A_{p,d}) = 1 + \frac{f(p^2)}{g(p^2)}$ for all $p \geq 3$;
\item (Gessel--Monsky, \cite{GesselMonsky}) $\lim_{p \to \infty} \ehk(A_{p,d}) = 1 + \En{d}/d!$;
\item $\ehk(A_{p,d}) > \ehk(A_{p,d+1})$ for fixed $p \geq 3$, i.e., 
the function $d \mapsto \ehk(A_{p,d})$ is decreasing. 
\end{enumerate}
\end{theorem*}

The first corollary significantly strengthens \cite[Theorem (B)]{Trivedi}, which shows that $\ehk(A_{p,d})$ is a rational function 
for $p \geq 3^{\lfloor (d-1)/2 \rfloor} (d-3)$ with the degrees 
of the denominator and numerator bounded by $d^{d+2}$. 
In fact, our bound is sharp since $\ehk (A_{p, 4}) = 1 + \frac{5p^2 + 3}{24p^2 + 12}$. The degree bound, immediately provides (Remark~\ref{remark algorithm}) an easy algorithm for computing the rational function $\ehk(A_{p,d})$ ($d$ fixed, $p \geq 3$ varies).

Our combinatorial approach is powerful enough to reprove the aforementioned results of Meng.  Utilizing an earlier version of this work and a deeper analysis of the Ehrhart polynomial of the Fibonacci polytope, Kahane (\cite{Kahane}) resolved Conjecture~\ref{conj WY 2}, thus recovering Meng's result \cite{Meng}. Furthermore, Appendix~\ref{AI} builds on the work \cite{CoxSteibRedner}
of Cox-Steib and Redner on the properties of matrices $\bT_n$ and $\bN_n$ to  show that  the sequence $p \mapsto \ehk(A_{p,d})$ is decreasing, thus reproving \cite{MengMonotonicity}.
By evaluating the rational function at $0$, it follows that for all $p \geq 3$ and $d \geq 4$
\[
1 + 1/d > \ehk (A_{p, d}) > 1 + \En{d}/d!.
\] 

\subsection*{Acknowledgements} 
We thank Kyle Petersen, and Akiyoshi Tsuchiya for suggestions and comments on an earlier draft of this paper, 
Akihiro Higashitani for helping us with Corollary~\ref{cor: rational function},
and Joachim Jelisiejew for suggesting to study the evaluation at $0$ of the rational function defined by $\ehk (A_{p, d})$. 
We thank the anonymous referees for suggestions that improved several parts of the manuscript.

The first author was partially supported by the NSF grant CCF-2302173. The third author was supported by the State Research Agency of Spain through the Ramon y Cajal fellowship RYC2020-028976-I funded by MCIN/AEI/10.13039/501100011033 and by FSE ``invest in your future'', and projects PID2021-125052NA-I00 and PID2024-156181NB-C31 funded by MICIU/AEI/10.13039/501100011033 and FEDER, EU. The last-named author was partially supported by JSPS Grant-in-Aid for Scientific Research (C) Grant Numbers, 24K06678.

\subsection*{AI use disclaimer}
The main body of this manuscript was produced without any use of AI.
Appendix~\ref{AI} was originally produced by ChatGPT Astra and then edited by Claude Fable and Opus. The resulting draft was altered and polished by the authors to their taste. 
The authors take full responsibility for the mathematical content and the final form of the article.

\section{The Fibonacci polytope and its Ehrhart polynomial}

We begin by fixing notation. We use $[n]$ to denote $\{1, \ldots, n\}$. If $P$ is an integer polytope, we use $|P|$ to denote the number of integer points in $P$ and $kP$ for its $k$-dilation.
It is known that the function $k \mapsto |kP|$ is given by a polynomial, which is called the Ehrhart polynomial and denoted by $\Ehr P(k)$. In this work we will need the following two families of polytopes.

\begin{definition}\label{def Fibonacci}
For $d \geq 1$, define the $d$-dimensional \emph{Fibonacci polytope} $F_d$ as 
the subset of $[0, 1]^d$ given by the inequalities
\[
x_{i} + x_{i+1} \leq 1,\;\; i = 1, \ldots, d-1 
\]
and the \emph{extended $d$-dimensional Fibonacci polytope} $E_d$ as 
the subset of $[-1, 1]^d$
given by the inequalities:
\[
|x_{i}| + |x_{i+1}| \leq 1,\;\;  i = 1, \ldots, d-1. 
\]
\end{definition}

In the terminology of \cite{Stanley, OhsugiTsuchiya1}, the (extended) Fibonacci polytope is the (enriched) chain polytope of the zigzag poset. 

\begin{remark}\label{Jacobsthal}
All integer points in the Fibonacci polytope are its vertices; by recursion one can show 
that the number of vertices is given by the Fibonacci numbers. 

The numbers of integer points $n(d)$ in the extended Fibonacci polytope of dimension $d$ form the Jacobsthal sequence
$(2^{d+2} - (-1)^{d+2})/3$ (\cite[A001045]{oeis}). Namely, it satisfies the linear recurrence $n(d+1) = n(d) + 2 n(d-1)$ with the standard initial conditions:
if we increase the dimension by adding $x_{d+1}$, then 
any integer point can be extended by $x_{d+1} = 0$. However, 
if $x_{d} = 0$, then the point can be also extended by $\pm 1$.

The number of vertices $v(d)$ of the extended Fibonacci polytope 
satisfies the linear recurrence $v(d) = 2v(d-2) + 2v(d-3)$ 
with the initial condition $v(1) = 2, v(2) = 4$ (\cite[A107383]{oeis}).
First, note that vertices correspond to words on the alphabet $\{-1, 0, 1\}$ such that 
there are no consecutive nonzero entries and no `$0$' can be replaced by `$1$'.
The last condition means that the sequence cannot start or end with `$00$' and cannot 
contain three consecutive `$0$'s. 
Such a sequence must therefore end with `$0\pm 1$' or `$0\pm 10$'. In the first case,
removing the ending gives a word in $v(d-2)$; in the second case, a word from $v(d-3)$. 
The recursion follows.  
\end{remark}

By \cite{Stanley} the Fibonacci polytope 
is affinely equivalent to the order polytope of the zigzag poset, defined as 
the set of points in $[0,1]^d$ such that $x_1 \geq x_2 \leq x_3 \geq \cdots $
by mapping $x_i \to y_i$ where $y_i = 1 - x_i$ if $i$ is odd and $y_i = x_i$ if $i$ is even. Thus, the two polytopes have equal Ehrhart polynomials and 
a result of Stanley (\cite{Stanley}) gives a canonical triangulation of the Fibonacci polytope. 

\begin{theorem}[Stanley]\label{t Stanley}
The order polytope of a poset $P$ on $[n]$ has a canonical triangulation indexed by linear extensions $P \to [n]$, i.e., permutations of $[n]$ that 
adhere to the partial order. 
\end{theorem}

Hence the order polytope of the zigzag poset is triangulated 
by \emph{alternating} permutations. Since the number of 
such permutations is the \emph{Euler up-down number} $\En{d}$, we obtain a formula for the volume (see also (\cite[Section 3.8]{StanleySurvey}).
Note that $\sum_{d \geq 0} \En{d}/d! \, x^d = \sec x + \tan x$ 
\cite{Andre1881}.

\begin{corollary}
The volume of the Fibonacci polytope of dimension $d$ equals $\En{d}/d!$.
\end{corollary}

To the best of our knowledge, there is no closed-form description of the Ehrhart polynomial of this polytope, but a recent paper of Coons and Sullivant \cite{CoonsSullivant} gives a combinatorial formula for the $h^*$ vector.

\begin{definition}
Let $\sigma$ be an alternating permutation in $S_n$. Its \emph{permutation statistic} $\swap(\sigma)$ is the number of $i < n$ such that $\sigma^{-1}(i)<\sigma^{-1}(i+1)-1$. Equivalently, $\swap(\sigma)$ is the number of $i < n$ such that $i$ is to the left of $i+ 1$ and swapping $i$ and $i+ 1$ in $\sigma$ yields another alternating permutation. 
\end{definition}

By \cite[Theorem~13]{CoonsSullivant} the Ehrhart series of the order polytope $Z_d$ of the zigzag poset is 
\[
\frac{\sum_{\sigma } t^{\swap(\sigma)}}{(1-t)^{d+1}}.
\]
over alternating permutations.
This result essentially follows from Theorem~\ref{t Stanley} after realizing that 
the adjacency of two simplices is given by a swap (\cite[Proposition~2.1]{CoonsSullivant}).
Note that the definition of swap is asymmetric to avoid double counting. 
Coons and Sullivant showed that the $h^*$-vector is symmetric and unimodal (see also \cite{PetersenZhuang} for a generalization). 

\begin{corollary}\label{c swap Ehrhart poly}
The swap statistic provides a formula for the Ehrhart polynomial of the Fibonacci polytope:
if $s_d(m) = \# \{\sigma \mid \swap (\sigma) = m\}$, then
\[
|kF_d|  = \sum_{m = 0}^{d-2} s_d(m) \binom{k+d-m}{d}
= \sum_{i = 0}^{d-2} (-1)^i \binom{k + d - i}{d - i} \left( 
\sum_{m = 0}^{d-2} \binom{m}{i} s_d(m)
\right ). 
\]
\end{corollary}
\begin{proof}
The first equality follows from the formula for the Ehrhart series. The second 
formula can be obtained either by manipulating the binomial coefficients or by 
using the inclusion-exclusion formula on the Stanley triangulation based on 
\cite[Proposition~2.1]{CoonsSullivant}: note that $\binom{k + d - i}{d - i}$ is the 
Ehrhart polynomial of a $(d-i)$-dimensional unit simplex.
\end{proof}

\begin{remark}\label{remark Kreweras}
The numbers $s_d(m)$ are recorded in \cite[A205497]{oeis}.
The table of $\sum_{m = 0}^{d-2} \binom{m}{i} s_d(m)$
is given in \cite[A079502]{oeis} based on a paper of Kreweras 
\cite[table on page 20]{Kreweras}. 
Kreweras denotes by $u^n_r$ the number of \emph{alternating surjective} maps $f\colon [n] \to [r]$, i.e., surjections satisfying the up-down condition $f(1) > f(2) < f(3) > \cdots $. 
The equality $u^n_{n-r} = \sum_{m = 0}^{n-2} \binom{m}{r} s_n(m)$
can be shown by constructing $f$ from a permutation $\pi \colon [n] \to [n]$
with $r$ chosen swaps by identifying the values $\pi(i)$ and $\pi(j)$ with $i < j$
if they form a swap. This reduces the image to $n - r$ numbers. The swap condition guarantees that the map is still alternating. 
\end{remark}

\begin{remark}[Petersen]
By Corollary~\ref{c swap Ehrhart poly}, we have $|F_d| = (d+1) + s_d(1)$. Thus 
$s_d(1) = F(d+1) - (d+1)$ (a shift of \cite[A001924]{oeis}), where $F(n)$ is the Fibonacci sequence with the initial condition $F(0) = F(1) = 1$. 
\end{remark}

\begin{definition}
A lattice polytope $P$ is \emph{reflexive} if there is an interior lattice point $x$ such that every facet has lattice distance one from $x$. In this case, $x$ is the unique interior lattice point of $P$. A lattice polytope $P$ is \emph{Gorenstein} if there is some positive integer $r$ such that $rP$ is reflexive. In this case, $r$ is uniquely determined and called the index of $P$.
\end{definition}

\begin{lemma}\label{lem: Gorenstein}
The Fibonacci polytope is a Gorenstein polytope of index $3$ and 
the extended Fibonacci polytope is a Gorenstein polytope of index $1$. 

In particular, the sequence $s_d(0), \ldots, s_d(d-2)$ is symmetric.
\end{lemma}
\begin{proof}
By the proof of \cite[Theorem~32]{CoonsSullivant}
the Fibonacci polytope $F_d$
is a Gorenstein polytope of index $3$ (i.e., $a$-invariant $3$).
Similarly, the extended Fibonacci polytope $E_d$ is a Gorenstein polytope of index $1$ (see, also \cite[Theorem~0.1]{OhsugiTsuchiya1}). 
Hence, the $h^*$-polynomial of $F_d(x)$ 
is a symmetric polynomial of degree $d-2$. 
\end{proof}

\begin{lemma}\label{l facet number}
We have a relation 
\[
\sum_{m = 0}^{d-2} ms_d(m) = \En{d}\left (\frac{d}2 - 1\right).
\]
\end{lemma}
\begin{proof}
It was shown in \cite[Theorem~4.1]{CoonsSullivant} that the 
sequence $s_d(0), \ldots, s_d(d-2)$ is symmetric. 
Hence $is_d(i) + (d-2 - i)s_d(d-2 -i) = (d/2 - 1) (s_d(i) + s_d(d-2 -i))$.
Thus, the lemma easily follows from the relation $\sum_{m = 0}^{d-2} s_d(m) = \En{d}$.
\end{proof}

We are now able to expand the Ehrhart polynomial of the Fibonacci polytope to the next term. 

\begin{corollary}\label{c Ehrhart expand}
The Ehrhart polynomial of the $d$-dimensional Fibonacci polytope is 
\[
\Ehr{P_d}(k)  = \frac{\En{d}}{d!} k^d + \frac{3}{2} \frac{\En{d}}{(d-1)!} k^{d - 1}
+ \frac{1}{(d-2)!}\left(\En{d}\frac{-3d^2 + 17d + 2}{24} + \sum_{m = 0}^{d-2} \binom{m}{2} s_d(m) \right) k^{d - 2} + O(k^{d - 3}).
\]

In particular, 
\[
\Ehr{P_d}(k) \left(x - \frac 32 \right) =   \frac{\En{d}}{d!} x^d + 
\frac{1}{(d-2)!}\left(\En{d}\frac{-3d^2 + 17d - 25}{24} + \sum_{m = 0}^{d-2} \binom{m}{2} s_d(m) \right) x^{d - 2} + O(x^{d - 3}).
\]
\end{corollary}
\begin{proof}
By Lemma~\ref{lem: Gorenstein}, $s_d(0), \ldots, s_d(d-2)$ is symmetric. 
Hence $is_d(i) + (d-2 - i)s_d(d-2 -i) = (d/2 - 1) (s_d(i) + s_d(d-2 -i))$.
Thus, the formula for the coefficient at $k^{d-1}$ easily follows from the relation $\sum_{m = 0}^{d-2} s_d(m) = \En{d}$.

We now compute the coefficient at $k^{d - 2}$.
By Corollary~\ref{c swap Ehrhart poly} it can be written as 
\[
\frac{\sum_{1 \leq i < j \leq d} ij}{d!} \times
\sum_{m = 0}^{d-2} s_d(m)
- \frac{\binom{d}{2}}{(d-1)!} \times \sum_{m = 0}^{d-2} m s_d(m) +
\frac{1}{(d - 2)!} \times
\sum_{m = 0}^{d-2} \binom{m}{2} s_d(m).
\]
Since $\sum_{1 \leq i < j \leq d} ij = \frac{(d-1)d(d+1)(3d+2)}{24}$, 
it now remains to use Lemma~\ref{l facet number}
to compute that
\[
\frac{\sum_{1 \leq i < j \leq d} ij}{d!}
\sum_{m = 0}^{d-2} s_d(m)
- \frac{\binom{d}{2}}{(d-1)!} \sum_{m = 0}^{d-2} m s_d(m)
= \frac{\En{d}}{(d-2)!} \frac{(d+1)(3d + 2)}{24}
- \frac{\En{d}}{(d-2)!} \frac{d(d - 2)}{4}. 
\]

Now the conclusion can be verified straightforwardly. 
\end{proof}

\begin{remark}
The coefficients of 
the Ehrhart polynomial of the Fibonacci polytope are positive
by \cite[Corollary~5.2]{FMP}. 
\end{remark}

\begin{remark}
By \cite[Theorem~0.3]{OhsugiTsuchiya1} (see also \cite[Theorem~4.6]{Petersen} and \cite{OkadaTsuchiya}) the $h^*$-vector of the extended $d$-dimensional Fibonacci polytope is $(x+1)^d p_d (4x(x+1)^{-2})$, where 
$p_d(x)$ is the left peak polynomial of the zigzag posed with natural labeling\footnote{A poset $([n], \prec)$ is naturally labeled if $x \prec y$ implies $x \leq y$.}. Explicitly, fix a natural labeling
$\omega\colon Z_d\to[d]$ and encode
a linear extension $f$ of $Z_d$ by the word $w=w_1\cdots w_d$ with $w_k=\omega(f^{-1}(k))$,
listing the labels of the elements in the order prescribed by $f$. Then
$p_d(x)=\sum_f x^{\mathrm{pk}^{(\ell)}(w)}$, where the left peak number
$\mathrm{pk}^{(\ell)}(w)$ of \cite{Petersen} is the number of $k\in[d-1]$ with
$w_{k-1}<w_k>w_{k+1}$, with the convention $w_0=0$. As a
statistic on alternating permutations the left peak number therefore depends
on $\omega$, and we do not know an intrinsic description of it.

\end{remark}

\section{The Han--Monsky algorithm and its matrix interpretation}
\subsection{Han--Monsky algorithm for computing the Hilbert--Kunz multiplicity}

In \cite{HanMonsky} Han and Monsky provided an approach to the calculation of Hilbert--Kunz multiplicity for a special class of local rings. 
The algorithm utilizes the \emph{Han--Monsky representation ring}, which is a certain Grothendieck ring.  Let $k$ be a fixed field and consider the subcategory $\modt(k)$ of finitely generated $k[T]$-modules with a nilpotent action of the variable $T$. 
On pairs $(M, N)$ of such modules there is a natural equivalence relation: 
$(M_1, N_1) \equiv (M_2, N_2)$ if and only $M_1 \oplus N_2 \cong M_2 \oplus N_1$. 
The equivalence class of a pair $(M, N)$ will be denoted as $M - N$. 

\begin{definition}
The representation ring $\Gamma_k$ consists of the equivalence classes $M - N$ with  
the operations of direct sum and tensor product:
\[
(M_1 - N_1) \times (M_2 - N_2) := 
\left [(M_1 \otimes_k M_2) \oplus (N_1 \otimes_k N_2) \right ]
-\left [(M_1 \otimes_k N_2) \oplus (N_1 \otimes_k M_2) \right ].
\]
The equivalence class $k - 0$ is the identity element of $\Gamma_k$.
\end{definition}

Since $k[T]$ is a PID, any $M \in \modt(k)$ decomposes as 
$M = \oplus_{i = 1}^{\infty} k[T]/(T^i)^{\oplus a_i}$. 
This shows that $\Gamma_k$ is a free $\mathbb Z$-module with the basis $\delta_i = k[T]/(T^i) - 0$, $i \in \mathbb{Z}_{> 0}$.
The function $D_k$ is defined by setting 
\[
D_k(M) := \dim_k M/TM = \sum_{i = 1}^\infty a_i
\] 
and extending it to $\Gamma_k$ by linearity. Thus 
$D_k$ is the sum of the coordinates in the basis $\{\delta_i\}$.

The following result is a restatement of \cite[Lemma 5.6]{HanMonsky}.

\begin{theorem}\label{thm HanMonsky}
For a given field $k$ of positive characteristic $p > 0$, define $$R := k[[x_0, \ldots, x_d]]/(x_0^{n_0} + \cdots + x_{d}^{n_d}).$$
For a fixed integer $e \geq 0$ and each $i =0, \ldots, d$, define the remainder $r_i$ determined by $p^e = a_in_i  + r_i$, $0 \leq r_i < n_i$.
In the above notation, 
\[
\dim_k R/(x_0^{p^e}, \ldots, x_d^{p^e}) = D_k \left( 
\prod^{d}_{i = 0}  (n_i - r_i)\delta_{a_i} + r_i \delta_{a_i + 1} \right).
\]
\end{theorem}
\begin{proof}
Let us sketch the idea of the proof. For a positive integer $n$, consider $k[x]/(x^{p^e})$ as a $k[T]$-module with the action $Tf  := x^{n}f$. Denote this module as $M_n$. 
Consider the division with the remainder $p^e =an + r$. 
Then the submodule $k[T] x^i = k\langle x^i, x^{i + n}, \ldots \rangle \subset M$ is isomorphic to $k[T]/(T^{a + 1})$ if $0 \leq i < r$ and to $k[T]/(T^a)$ if $r \leq i < n$.
This gives a direct sum decomposition 
\[
M_n = \oplus_{0 \leq i < n} k[T]x^i = [k[T]/(T^{a + 1})]^{\oplus r} \oplus [k[T]/(T^{a})]^{\oplus n-r}
\]
as a $k[T]$-module.

Furthermore,
$M_{n_0} \otimes_k \cdots \otimes_k M_{n_d} \cong k[x_0, \ldots, x_d]/ (x_0^{p^e}, \ldots, x_d^{p^e})$ 
with $T \mapsto x_0^{n_0} + \cdots + x_d^{n_d}$,
so $\dim_k R/(x_0^{p^e}, \ldots, x_d^{p^e}) = D_k(M_{n_0} \otimes_k \cdots \otimes_k M_{n_d})$. 
\end{proof}

The difficulty of using the theorem is the decomposition of the product of basis elements $\delta_i \delta_j$ as a linear combination of the basis elements. Han and Monsky 
were able to provide the needed multiplication rules. In particular, their work gives the following recipe.

\begin{corollary}\label{cor HanMonsky}
Let $k$ be a field of characteristic $p \geq 3$ and $a = (p-1)/2$.
Then 
\[
\ehk(A_{p,d}) := 
\ehk(k[[x_0, \ldots, x_d]]/(x_0^2 + \cdots + x_d^2)) = 1 + \frac{D_k((\delta_a + \delta_{a+1})^{d+1}) - p^d}{p^d - (-1)^{a(d+1)} D_k((\delta_{a+1} - \delta_{a})^{d + 1})}.
\]  
\end{corollary}
\begin{proof}
Essentially, this is explained in \cite[Example 2, page 134]{HanMonsky}.
Since $p$ is odd, $\mu = 1$ can be used in \cite[Theorem~5.3]{HanMonsky}, 
so \cite[Theorems~5.5, 5.7]{HanMonsky} will show that 
\[
\dim_k A_{p,d}/(x_0^{p}, \ldots, x_d^{p}) = p^{d} \ehk(A_{p,d})
+ (1 - \ehk(A_{p,d})) D_k((-1)^{a(d+1)} (\delta_{a+1} - \delta_{a})^{d+1})
\]
and we solve the equation for the Hilbert--Kunz multiplicity.

\end{proof}

\subsection{A linear algebra approach}\label{ss: HM linear algebra}
In order to use Corollary~\ref{cor HanMonsky} we need some of the multiplication rules in the representation ring $\Gamma_k$.  
Han and Monsky found it more convenient to work with another basis of $\Gamma_k$ instead of $\{\delta_i\}$. We set $\lambda_0 = \delta_1$
and $\lambda_i = (-1)^{i} (\delta_{i+1} - \delta_{i})$, 
so that $\delta_i = \lambda_0 - \lambda_1 + \cdots + (-1)^{i-1} \lambda_{i-1}$. 
Then, by definition, $D_k(M - N)$ is the coefficient at $\lambda_0$ of the decomposition of
$M- N$ in the basis $\{\lambda_i\}$. 

The following is an easy restatement of \cite[Theorem~2.5]{HanMonsky}.

\begin{theorem}[Han--Monsky]
Let $k$ be a field of characteristic $p > 0$ and 
$0 \leq i \leq j \leq p -1$. Then
\begin{enumerate}
\item if $i + j \leq p -1$ then $\lambda_i\lambda_j = \sum_{k = j - i}^{j + i} \lambda_k$.
\item if $i + j \geq p$ then $\lambda_i\lambda_j = \lambda_{p - 1 - i} \lambda_{p-1-j}$.
\end{enumerate}
\end{theorem}

It follows immediately that $\{\lambda_i\}_{i = 0}^{p-1}$ generate a subring $\Gamma_{k, p}$ and, by Corollary~\ref{cor HanMonsky}, we may compute the Hilbert--Kunz multiplicity $\ehk(A_{p,d})$ by studying multiplication matrices of elements in this subring. 

\noindent\textbf{Notation.}
The $(1,1)$-entry of a matrix $M$ is denoted $[M]_{1,1}$.
For an element $a \in \Gamma_{k, p}$ we will denote by
$\Mat(a)$ the $p\times p$ matrix describing the decomposition of $a\lambda_j$ in the basis 
$\{\lambda_i\}_{i = 0}^{p-1}$. The map $a \mapsto \Mat(a)$ is a ring homomorphism. 

\begin{example}
\[
\Mat(\lambda_1) = 
\begin{bmatrix}
0 & 1 & 0 & 0 & \cdots & 0 & 0 & 0\\
1 & 1 & 1 & 0 & \cdots & 0 & 0 & 0\\
0 & 1 & 1 & 1 & \cdots & 0 & 0 & 0\\
0 & 0 & 1 & 1 & \cdots & 0 & 0 & 0\\
\vdots & \vdots & \vdots & \vdots & \vdots & \vdots & \vdots & \vdots\\
0 & 0 & 0 & 0 & \cdots & 1 & 1 & 0\\
0 & 0 & 0 & 0 & \cdots & 1 & 1 & 1\\
0 & 0 & 0 & 0 & \cdots & 0 & 1 & 0
\end{bmatrix}
\text{ and }
\Mat(\lambda_2) = 
\begin{bmatrix}
0 & 0 & 1 & 0 & \cdots & 0 & 0 & 0\\
0 & 1 & 1 & 1 & \cdots & 0 & 0 & 0\\
1 & 1 & 1 & 1 & \cdots & 0 & 0 & 0\\
0 & 1 & 1 & 1 & \cdots & 0 & 0 & 0\\
\vdots & \vdots & \vdots & \vdots & \vdots & \vdots & \vdots & \vdots\\
0 & 0 & 0 & 0 & \cdots & 1 & 1 & 1\\
0 & 0 & 0 & 0 & \cdots & 1 & 1 & 0\\
0 & 0 & 0 & 0 & \cdots & 1 & 0 & 0
\end{bmatrix}.
\]
\end{example}

From Theorem~\ref{thm HanMonsky} we immediately get the following.
\begin{corollary}\label{cor matrices}
Let $k$ be a field of characteristic $p > 0$ and $R = k[[x_0, \ldots, x_d]]/(x_0^{n_0} + \cdots + x_{d}^{n_d})$.
For a fixed positive integer $n$, write $p = an + r$ with $0 \leq r < n$ and set $b = n-r$ and define the matrix 
\[
\mathbf{M}_n = 
\begin{bNiceMatrix}[nullify-dots]
n 	    &	 \Cdots  &n   	  & r 	  & 0 	  &\Cdots  & \Cdots & 0 \\
\Vdots & \Iddots  & \Iddots & b        & \Ddots& \Ddots & 	          & \Vdots\\
n  	    & \Iddots 	&\Iddots  & r		  & \Ddots& 		   & 		    &  \Vdots \\
r 	    & b 	      & r 		  &          &  \Ddots&   	 	  &           & 0 \\
0 	    &  	      &\Ddots	  & \Ddots& 		  & 	r	   & b         & r \\
\Vdots & \Ddots 	& \Ddots	  & \Ddots& r		  & \Iddots&  \Iddots&  n\\
\Vdots &  	      & 	  	  & \Ddots& b	  & \Iddots&  \Iddots &   \Vdots \\
0 	    &  \Cdots &   	   	  & 	0	  & 	r 	  &  n 	   & \Cdots  &  n\\
\end{bNiceMatrix},
\]
where the ``triangular corners'' of `$n$' have length $a$,
the ``triangular corners'' of `$0$' have length $p-a-1$, the  
the middle rectangle has alternating `$r$' and `$b$'.
Then, 
$\dim_k R/(x_0^{p}, \ldots, x_d^{p}) = \prod^{d}_{i = 0} \left( (n_i - r_i)\delta_{a_i} + r_i \delta_{a_i + 1} \right)$
is the $(1,1)$-entry of the matrix product $\mathbf{M}_{n_0} \mathbf{M}_{n_1} \cdots \mathbf{M}_{n_d}$.
\end{corollary}
\begin{proof}
Since 
$(n - r)\delta_{a} + r \delta_{a + 1} = 
n(\lambda_0 - \lambda_1 + \cdots + (-1)^{a-1} \lambda_{a-1}) + r(-1)^a \lambda_a$, $\Mat((n - r)\delta_{a} + r \delta_{a + 1})$ is 
$\mathbf{M}_n$ with alternating signs:
\[
\begin{bNiceMatrix}[nullify-dots]
n 	    &	 -n   &  \Cdots  & (-1)^{a-1} n   	  & (-1)^{a} r 	& 0 		&\Cdots	& \Cdots & 0 \\
-n 	    & n   &   \cdots	  &	 (-1)^{a}r  	  &  (-1)^{a+1}b  	&  (-1)^{a+2} r   	& 0 & \cdots	 &0 \\
\vdots & \vdots & \vdots & \vdots & \vdots & \vdots & \vdots & \vdots & \vdots 
\end{bNiceMatrix}.
\]
Since the signs only depend on the parity and are identical in the matrices, 
\[
\left [\prod_{i = 0}^d
\mathbf{M}_{n_i}\right ]_{1,1} 
= \left [\prod^{d}_{i = 0} \Mat \left( (n_i - r_i)\delta_{a_i} + r_i \delta_{a_i + 1} \right) \right ]_{1,1} .
\]
\end{proof}

We now obtain the matrix form of Corollary~\ref{cor HanMonsky}.

\begin{corollary}\label{cor Matrix Form}
For an integer $a \geq 1$, define the square matrices $\bT_a$ and $\bN_a$ of size $(2a+1)$
\[
\bT_a = 
\begin{bNiceMatrix}[nullify-dots]
2 	    &	 \Cdots   &   2	 & 1   	  & 0  	&  \Cdots & 0  \\
\Vdots &	 \Iddots  & \Iddots& \Vdots & \Ddots&  \Ddots & \Vdots  \\
2 	    &	 \Iddots  &           &    	  &         &  \Ddots & 0  \\
1	    &	\Cdots  	& 		 & 1   	  &\Cdots&  		   & 1  \\
0	    &	  	       & 		 &\Vdots  &   	&  		   & 2  \\
\Vdots &	 \Ddots  	 & \Ddots&    	  &\Iddots&\Iddots   &\Vdots  \\
0	    &	\Cdots  	 & 0		 &    1	  & 2  	&\Cdots    & 2  \\
\end{bNiceMatrix} \text { and }
\bN_a = 
\begin{bNiceMatrix}[nullify-dots]
0 	    &	 \Cdots   &   0	 & 1   	  & 0  	&  \Cdots & 0  \\
\Vdots &	 \Iddots  & \Iddots& \Vdots & \Ddots&  \Ddots & \Vdots  \\
0 	    &	 \Iddots  &           &    	  &         &  \Ddots & 0  \\
1	    &	\Cdots  	& 		 & 1   	  &\Cdots&  		   & 1  \\
0	    &	  	       & 		 &\Vdots  &   	&  		   & 0  \\
\Vdots &	 \Ddots  	 & \Ddots&    	  &\Iddots&\Iddots   &\Vdots  \\
0	    &	\Cdots  	 & 0		 &    1	  & 0  	&\Cdots    & 0  \\
\end{bNiceMatrix}.
\]
Let $k$ be a field of characteristic $p \geq 3$.
If $a = (p - 1)/2$, then
\[
\ehk(A_{p,d}) = 
\ehk(k[[x_0, \ldots, x_d]]/(x_0^2 + \cdots + x_d^2)) = 1 + \frac{[\bT_a^{d+1}]_{1,1} - p^d}{p^d - [\bN_a^{d+1}]_{1,1}}.
\]
\end{corollary}

\subsection{Main results}

The following lemma is a partial case and a warm-up result for the main proof. 

\begin{lemma}\label{lemma: first Fibonacci}
Fix a positive integer $n$ and define the $(2n+1) \times (2n+1)$-matrix 
\[
\bZ = 
\begin{bNiceMatrix}[nullify-dots]
1 	    &	 \Cdots   &   1	 & 0   	  & 1  	&  \Cdots & 1  \\
\Vdots &	 \Iddots  & \Iddots& \Vdots & \Ddots&  \Ddots & \Vdots  \\
1 	    &	 \Iddots  &           &    	  &         &  \Ddots & 1  \\
0	    &	\Cdots  	& 		 & 0   	  &\Cdots&  		   & 0  \\
1	    &	  	       & 		 &\Vdots  &   	&  		   & 1  \\
\Vdots &	 \Ddots  	 & \Ddots&    	  &\Iddots&\Iddots   &\Vdots  \\
1	    &	\Cdots  	 & 1		 &    0	  & 1  	&\Cdots    & 1  \\
\end{bNiceMatrix}.
\]
Then $[\bZ^{d+1}]_{1,1} = 2^d\Ehr{F_d} (n-1)$, 
where $F_d$ denotes the $d$-dimensional Fibonacci polytope.
\end{lemma}
\begin{proof}
By the definition of the matrix product, we may write
\[
[\bZ^{d+1}]_{1,1} = \sum_{1 \leq i_1, \ldots, i_d \leq 2n+1} [\bZ]_{1, i_1} [\bZ]_{i_1, i_2}\cdots [\bZ]_{i_d, 1}.
\]
Let $
S = \left\{ (1, i_1, \ldots, i_d, 1) \mid i_k \in [n], \quad i_k + i_{k+1} \leq n+1 \text{ for all } k \right \} 
$ and
observe that the map 
\[
S \ni (1, i_1, \ldots, i_d, 1) \mapsto (i_1-1, i_2-1, \ldots, i_d-1)
\]
is a bijection with the integer points in the $(n-1)$th dilation of the $d$-dimensional Fibonacci polytope $F_d$. Thus
it suffices to verify that the number of tuples $(1, i_1, \ldots, i_d, 1)$ such that $[\bZ]_{1, i_1} \cdots [\bZ]_{i_d, 1} \neq 0$ is $2^d |S|$.

Now, we note that $S$ corresponds to the products $[\bZ]_{1, i_1} \cdots [\bZ]_{i_d, 1}$ where 
each entry is in the top-left triangular block. 
On the other hand, if $[\bZ]_{i, i_{1}} \neq 0$ then it may belong to either the top-left (which $S$) or bottom-left triangular blocks. 
These choices are equivalent and, after making either choice, we 
will similarly have two possible corners for $i_2$ and so on. 
\end{proof}

\begin{theorem}\label{thm Fibonacci}
Let $F_d$ denote the $d$-dimensional Fibonacci polytope. 
Consider the matrix $\bT_n$ defined in Corollary~\ref{cor Matrix Form}.
Then for any positive integer $n$, we have
\[
[\bT_{n}^{d+1}]_{1,1} = (2n + 1)^{d} + 2^d\Ehr{F_d} (n-1).
\]
\end{theorem}
\begin{proof}
To start, observe that $(2n+1)^{d} = [{\mathbf I}_n^{d+1}]_{1,1}$, where ${\mathbf I}_n$ is the $(2n+1) \times (2n+1)$ square matrix such that $[{\mathbf I}_n]_{i, j} = 1$ for all $i, j$. 
We will work with \emph{nonzero} products in the expression 
\[
[\bT_n^{d+1}]_{1,1} = \sum_{1 \leq i_1, \ldots, i_{d} \leq 2n+1} [\bT_n]_{1, i_1} [\bT_n]_{i_1, i_2} \cdots [\bT_n]_{i_{d}, 1}.
\]
In the proof, we will identify products $[\bT_n]_{1, i_1} [\bT_n]_{i_1, i_2} \cdots [\bT_n]_{i_{d}, 1}$ with tuples of indices 
\[
(\overline{i_1}, \ldots, \overline{i_d}) := ( i_1-n-1, \ldots, i_{d}-n-1) \in \{-n, \ldots, n\}^{d}
\] 
and will similarly use tuples $(\overline{j_0}, \ldots, \overline{j_{d+1}}) \in \{-n, \ldots, n\}^{d+2}$ to label products of elements of $\mathbf{I}_n$. 
In correspondence with the shape of $\bT_n$, 
let us introduce $2$ regions in $[-n, n]^2$: the first is the middle rhombus $\mathcal R$, $|x| + |y| \leq n$, and the second, $\mathcal C$, consists of 
$(x, y)$ such that $|x| + |y| > n$ and $xy > 0$ (the nonzero corners of $\bT_n$). 
Note that $[\bT_n]_{1, i_1} [\bT_n]_{i_1, i_2} \cdots [\bT_n]_{i_{d}, 1} \neq 0$
if and only if each consecutive pair of entries in 
$(-n, i_1 -n -1, \ldots, i_d-n-1, -n)$ 
is either in $\mathcal{R}$ or $\mathcal{C}$. 

Suppose that $[\bT_n]_{1, i_1} [\bT_n]_{i_1, i_2} \cdots [\bT_n]_{i_{d}, 1} \neq 0$.
Due to the structure of $\bT_n$, the product must be equal to 
$2^l$ for some $0 \leq l \leq d+ 1$. 
The number of products (equivalently, tuples) which are equal to $2^{d+1}$ is counted easily: 
in this case $[\bT_n]_{i_k, i_{k + 1}}$ must always belong to the upper left corner of $\bT_n$, so there are $|(n-1)F_d|$ such tuples by the proof of Lemma~\ref{lemma: first Fibonacci}.

We now may assume that $l \neq d+1$. We give a procedure $P$ that from a given tuple $(\overline{i_1} \ldots, \overline{i_{d}})$ will produce $2^l$ distinct tuples $(\overline{j_0}, \ldots, \overline{j_{d+1}})$. 
\begin{procedure}
Set $\overline{i_0} := -n$ and $\overline{i_{d+1}} := -n$. 
Let $k_0$ be the smallest integer $k$ such that $(\overline{i_k}, \overline{i_{k+1}}) \in \mathcal{R}$  (i.e., $[\bT_n]_{i_k, i_{k+1}} = 1$). 
We now analyze all $i_0, i_1, \ldots, i_{d+1}$ starting from the left: 
\begin{enumerate}
\item set $\overline{j_{d+1}} = \overline{j_0} = -n$;
\item if $m < k_0$ and $(\overline{i_{m}}, \overline{i_{m+1}}) \in \mathcal{C}$, then put two values,  $\overline{j_{m+1}} = \pm \overline{i_{m+1}}$; 
\item if $m > k_0$ and $(\overline{i_{m}}, \overline{i_{m+1}}) \in \mathcal{C}$, then put two values, $\overline{j_{m}} = \pm \overline{i_{m}}$;
\item if $m > k_0$ and $(\overline{i_{m}}, \overline{i_{m+1}}) \in \mathcal{R}$, then $\overline{j_{m}} = \overline{i_{m}}$.
\end{enumerate}
\end{procedure}
\begin{claim}
The map $P$ has the following properties:
\begin{itemize}
\item[(a)] the image of $P$ covers all tuples $(-n, \overline{j_1}, \ldots,\overline{ j_{d}}, -n)$
with at least one consecutive pair of entries in $\mathcal R$;
\item[(b)] if $(\overline{i_1}, \ldots, \overline{i_{d}}) \neq (\overline{i_1}', \ldots, \overline{i_{d}}')$ then $P(\overline{i_1}, \ldots, \overline{i_{d}}) \cap P(\overline{i_1}', \ldots, \overline{i_{d}}') = \emptyset$.
\end{itemize}
\end{claim}
\begin{proof}
First, since $P$ preserves absolute values by the construction, for  
a tuple $(-n, \overline{j_1}, \ldots, \overline{j_{d}}, -n) \in P(\overline{i_1}, \ldots, \overline{i_{d}})$ and all $m = 0, \ldots, d+1$, we have 
$(\overline{j_m}, \overline{j_{m + 1}}) \in \mathcal{R}$ if and only if 
 $(\overline{i_m}, \overline{i_{m + 1}}) \in \mathcal{R}$.
Note that the distinguished entry $(\overline{i_{k_0}}, \overline{i_{k_0 + 1}}) \in \mathcal{R}$ is no different, since the values $\overline{i_{k_0}}$ and $\overline{i_{k_0+1}}$ are operated with while checking $(\overline{i_{k_0-1}}, \overline{i_{k_0}})$ and $(\overline{i_{k_0+1}}, \overline{i_{k_0+2}})$. In particular, the image of $P$ contains only tuples with at least one entry in the rhombus $\mathcal R$. 

It suffices to show that, for a tuple $(-n, \overline{j_1}, \ldots, \overline{j_{d}}, -n)$, there is a unique tuple $(\overline{i_1}, \ldots, \overline{i_{d}})$ such that  
$[\bT_n]_{1, i_1} \cdots [\bT_n]_{i_{d}, 1} \neq 0$ and $(-n, \overline{j_1}, \ldots, \overline{j_{d}}, -n) \in P(\overline{i_1}, \ldots, \overline{i_{d}})$. 
Namely, we first recover $k_0$ as the smallest integer $m$ such that $(\overline{j_m}, \overline{j_{m + 1}}) \in \mathcal{R}$ due to the preservation of the rhombus $\mathcal{R}$. Since we require $[\bT_n]_{1, i_1} \cdots [\bT_n]_{i_{d}, 1}$ to be nonzero, 
\[
(-n, \overline{i_1}), \ldots, (\overline{i_{k_0-1}}, \overline{i_{k_0}}) \in \mathcal{C}.
\]
Thus, we set $\overline{i_{m}} = - |\overline{j_{m}}|$ for $m = 1, \ldots, k_0$.
Since $\overline{i_{k_0 + 1}}$ can be of any sign, we approach it starting from the tail.
Set $\overline{i_{d+1}} = -n$. Now, by steps (3) and (4) of the procedure, $\overline{i_{m+1}}$ and $(\overline{j_{m}}, \overline{j_{m+1}})$ uniquely determine $\overline{i_{m}}$: if $(\overline{j_{m}}, \overline{j_{m+1}}) \in \mathcal R$ then we set $\overline{i_{m}} = \overline{j_m}$; otherwise, 
we set 
\[
\overline{i_{m}} = \sign (\overline{j_{m+1}}) |\overline{j_m}|,
\] 
ensuring that $(\overline{i_{m}}, \overline{i_{m+1}}) \in \mathcal{C}$.
We proceed this way until reaching $m = k_0 + 1$ in which case all values are determined. 
By the construction, $(-n, \overline{j_1}, \ldots, \overline{j_{d}}, -n) \in P(\overline{i_1}, \ldots, \overline{i_{d}})$.
\end{proof}

The claim immediately implies that $[\bT_n^{d + 1}]_{1,1} - 2^{d+1} \Ehr{F_d}(n-1)$ is the number of all 
tuples $(-n, \overline{j_1}, \ldots, \overline{j_d}, -n)$ that contain at least one consecutive pair of entries in the rhombus $\mathcal {R}$. 
It remains to count the number of tuples $(\overline{j_1}, \ldots, \overline{j_d})$ 
such that 
\[
(-n, \overline{j_1}), (\overline{j_1}, \overline{j_2}), \ldots, (\overline{j_{d-1}}, \overline{j_d}), (\overline{j_d}, -n) \notin \mathcal{R}.
\]
Lemma~\ref{lemma: first Fibonacci} shows that the number of such tuples is $2^d\Ehr{F_d}(n-1)$.
It now follows that 
\[(2n+1)^d = [\mathbf I_n^{d+1}]_{1,1} = 
[\bT_n^{d + 1}]_{1,1} - 2^{d}\Ehr{F_d}(n-1).
\]
\end{proof}

\begin{corollary}\label{cor main Ehrhart}
Let $A_{p, d} = k [[x_0, \ldots, x_d]]/(x_0^2 + \cdots + x_d^2))$ where $k$ is a field of characteristic $p \geq 3$. If $F_d$ and $E_d$ denote the Fibonacci and extended Fibonacci polytopes, then 
\[
\dim_{k} \left(  \frac{A_{p, d}}{(x_0^{p^n}, \ldots, x_d^{p^n})} \right )
= \left (1 + \frac{2^d \Ehr{F_d} \left (\frac{p-3}{2} \right) }{p^d - \Ehr{E_{d-2}} \left (
\frac{p-1}{2} \right )} \right )p^{nd} - 
\frac{2^d \Ehr{F_d} \left (\frac{p-3}{2} \right) }{p^d - \Ehr{E_{d-2}} \left (
\frac{p-1}{2} \right )} \Ehr{E_{d-2}}\left (
\frac{p-1}{2} \right )^n.
\]
In particular, 
$
\ehk(A_{p, d})
= 1 + \frac{2^d \Ehr{F_d} \left (\frac{p-3}{2} \right) }{p^d - \Ehr{E_{d-2}} \left (
\frac{p-1}{2} \right )}.
$
\end{corollary}
\begin{proof}
Set $a = (p-1)/2$.
By Corollary~\ref{cor Matrix Form},
\[
\ehk(k[[x_0, \ldots, x_d]]/(x_0^2 + \cdots + x_d^2))
= 1 + \frac{[\bT_a^{d+1}]_{1,1} - (2a+1)^d}{(2a+1)^d - [\bN_a^{d+1}]_{1,1}}.
\]
By Theorem~\ref{thm Fibonacci} the numerator has the required form. 
As for the denominator, we interpret 
\[
[\bN_a^{d+1}]_{1,1} = \sum_{-a \leq i_1, \ldots, i_d \leq a} [\bN_a]_{1, a+1 + i_1}\cdots [\bN_a]_{a+1 + i_d, 1}.
\]
Since we need all $[\bN_a]_{i,j} \neq 0$, we must have $i_1 = i_d = 0$ and the remaining indices clearly correspond to the integer points 
in the $a$-dilation of the extended Fibonacci polytope. 
The assertion now follows from \cite{HanMonsky} as explained in Corollary~\ref{cor HanMonsky}.
\end{proof}

\begin{example}
By Corollary~\ref{cor main Ehrhart} and Remark~\ref{Jacobsthal}, we see that 
\[
\ehk(\mathbb{F}_3[[x_0, \ldots, x_d]]/(x_0^2 + \cdots + x_d^2)) = 1 + \frac{2^d3}{3^{d+1} - 2^d + (-1)^d}.
\]
By Corollary~\ref{cor main Ehrhart}, after cancelling the common factor
\[
\dim_{k} \left(  \frac{A_{p, 5}}{(x_0^{p^n}, \ldots, x_5^{p^n})} \right )
= \frac{17p^2 + 12}{15p^2 + 10}\, p^{5n} - \frac{2p^2+2}{15p^2 + 10} \left(\frac{p^3+2p}{3} \right)^n.
\]

\end{example}

It immediately follows from Corollary~\ref{cor main Ehrhart}
that $p \mapsto \ehk(\mathbb{F}_p[[x_0, \ldots, x_d]]/(x_0^2 + \cdots + x_d^2))$ is a rational function of degree $d$. 
In fact, this function is even. 
The original idea for the following proof was suggested to us by Akihiro Higashitani. 

\begin{corollary}\label{cor: rational function}
The function 
$
p \mapsto 
\ehk(\mathbb{F}_p[[x_0, \ldots, x_d]]/(x_0^2 + \cdots + x_d^2))
$
is rational. Moreover, it can be written as $f(p^2)/g(p^2)$ 
where $f, g$ are polynomials of degree $\lfloor d/2 \rfloor - 1$. 
\end{corollary}
\begin{proof}
By Lemma~\ref{lem: Gorenstein} the $h^*$ polynomial $h_{F_d}(x)$  of $F_d(x)$ (resp., $h_{E_d}$ of $E_d$) 
is a symmetric polynomial of degree $d-2$ (resp., $d$). 
Therefore, the Ehrhart polynomials satisfy the equations
\begin{equation}\label{eq: ehr symmetry}
\Ehr{F_d}(x) = (-1)^{d} \Ehr{F_d} (-x-3) \text{ and }
\Ehr{E_d}(x) = (-1)^{d} \Ehr{E_d}(-x-1).
\end{equation}
It now follows from Corollary~\ref{cor main Ehrhart}
that the numerator and denominator of the rational function
\[
p \mapsto \ehk(\mathbb{F}_p[[x_0, \ldots, x_d]]/(x_0^2 + \cdots + x_d^2))  - 1
= \frac{2^d \Ehr{F_d} (p/2 - 3/2) }{p^d - \Ehr{E_{d-2}} (p/2 - 1/2)}
\]
are simultaneously even or odd functions in $p$, depending on whether $d$ is even or odd. Therefore, the entire expression is a rational function in $p^2$
given by polynomials of degree $\lfloor d/2 \rfloor$. 

In order to cut down the degree, we observe that the numerator and the denominator
are divisible by $p^2 - 1$. 
In the numerator, by symmetry in (\ref{eq: ehr symmetry}), it suffices to show that $\Ehr{F_d}(-1) = 0$.
This follows from the Ehrhart--Macdonald reciprocity since 
$F_d$ does not have interior lattice points. 
As for the denominator, it suffices to note that 
$\Ehr{E_d}(-1) = (-1)^d$ by the Ehrhart--Macdonald reciprocity
and $\Ehr{E_d}(0) = 1$ by the property of Ehrhart polynomials. 
\end{proof}

\begin{remark}\label{remark algorithm}
One can compute this rational function via polynomial interpolation using the first $d+1$ values for the numerator and denominator polynomials, obtained by evaluating $T_k^{d+1}$ and $N_k^{d+1}$ for $k = 1, \ldots, d+1$. 

Alternatively, the initial values of the Ehrhart polynomial of the Fibonacci polytope can be computed using a recursion in \cite[Proposition~3.17]{PetersenZhuang}, 
and the interpolation can be bypassed using the method of \cite[Remark~3.18]{PetersenZhuang}. Another option is to use the recursive algorithm due to Kreweras (\cite{Kreweras}, \cite[A079502]{oeis}).
\end{remark}

From Corollary~\ref{cor main Ehrhart} we recover an unpublished result of Gessel and Monsky (\cite{GesselMonsky} (see also \cite[Theorem~4.1]{BrennerLiMiller}) and \cite[Corollary~4.5]{CaminataShidelerTuckerZerman}) and improve its convergence estimate.

\begin{corollary}\label{cor convergence rate}
Let $k$ be a field of characteristic $p \geq 3$. Then
\[
\ehk(A_{p,d}) := \ehk(k[[x_0, \ldots, x_d]]/(x_0^2 + \cdots + x_d^2))
= 1 + \frac{\En{d}}{d!} + O(p^{-2}).
\]
\end{corollary}
\begin{proof}
By Corollary~\ref{c Ehrhart expand} 
\[
2^d \Ehr{F_d} \left(\frac {p-3}2\right ) = \frac{\En{d}}{d!} p^d  + O(p^{d-2}).
\]
\end{proof}

\begin{corollary}\label{cor: monotone function}
For any field $k$ of characteristic $p \geq 3$ and any $d \geq 0$, we have
\[
\ehk(A_{p,d}) > \ehk(A_{p,d+1}).
\]
\end{corollary}
\begin{proof}
We apply the formula in Corollary~\ref{cor main Ehrhart}. First, observe that $ \Ehr{F_{d+1}}(a) < (a+1) \Ehr{F_d}(a)$. This follows from the definition: 
we add $x_{d+1}$ such that $x_{d} + x_{d+1} \leq a$, so there are at most $(a+1)$ options but some are not viable. 
Hence 
\[
\frac{2^{d+1} \Ehr{F_{d+1}}(a-1)}{(2a+1)^{d+1}} < \frac{(2a) 2^d \Ehr{F_d}(a-1) }{(2a+1)^{d+1}}
< \frac{2^d\Ehr{F_d}(a-1)}{(2a+1)^{d}}.
\]

Similarly, $\Ehr{E_{d+1}}(a) < (2a+1) \Ehr{E_d}(a)$.
Thus  for $a = (p-1)/2$ we have
\begin{align*}
\ehk(A_{p,d+1})
&= 1 + \frac{2^{d+1} \Ehr{F_{d+1}}(a-1)}{(2a+1)^{d+1} - \Ehr{E_{d-1}}(a)}
= 1 + \frac{2^{d+1} \Ehr{F_{d+1}}(a-1)/(2a+1)^{d+1}}{1 - \Ehr{E_{d-1}}(a)/(2a+1)^{d+1}}\\
&< 1 + \frac{2^{d} \Ehr{F_{d}}(a-1)/(2a+1)^{d}}{1 - \Ehr{E_{d-2}}(a) /(2a+1)^{d}}
=\ehk(A_{p,d}).
\end{align*}
\end{proof}

\begin{remark}
The corollary also holds for $p = 2$ by a direct computation performed in \cite{CastilloRey}.
Notably, this result demonstrates that $\ehk(A_{p,d-1}) = \ehk(A_{p,d}/(L)) > \ehk(A_{p,d})$ for {\it every} linear form $L$. Thus, there is no analogue of \emph{superficial} elements for Hilbert--Kunz multiplicity
(recall that $\eh(R) = \eh(R/(L))$ when $L$ is superficial). 
\end{remark}

As the last corollary we strengthen an easy reduction for the Watanabe--Yoshida conjecture observed in \cite{EnescuShimomoto}.

\begin{corollary}
It suffices to prove Conjecture~\ref{conj WY 1} for isolated singularities. 
\end{corollary}
\begin{proof}
Suppose that the Watanabe--Yoshida conjecture holds for isolated singularities. 
Let $(R, \mf m)$ be a formally unmixed local ring of positive characteristic $p > 0$. Since Hilbert--Kunz multiplicity does not change upon completion, we may assume that $R$ is complete. Hence its regular locus is open. Let $\mf p \neq \mf m$ be a minimal prime of its singular locus.
Since $R$ is equidimensional, we have $\ehk(R) \geq \ehk(R_\mf p) \geq \ehk(A_{p,\dim R_\mf p}) > \ehk(A_{p,\dim R})$. 
\end{proof}

\begin{remark}
An extension of Corollary~\ref{cor main Ehrhart} to Fermat hypersurfaces of higher degree, $R_p = \mathbb{F}_p[[x_0, \ldots, x_d]]/(x_0^D + \cdots + x_d^D)$, must overcome the appearing modular behavior. 
First, in the Han--Monsky algorithm one has to use 
$\dim_{\mathbb{F}_p} R_p/(x_0^{p^n}, \ldots, x_d^{p^n})$, where $n \geq 1$ depends on the residue class of $p$ modulo $D$. The resulting matrices are not just more complicated, but also exhibit alternating patterns that depend on the residue class of $p$, similarly to $n = 1$ in Corollary~\ref{cor matrices}. In fact, already the Hilbert--Kunz multiplicity of $x_1^3 + \cdots + x_4^3$ is not a rational function of $p$ due to periodicity (cf. \cite[Theorem~7.1]{CaminataShidelerTuckerZerman}). 

On the other hand, since $\lim_{p \to \infty} p^{-d} \dim_{\mathbb{F}_p} R_p/(x_0^{p}, \ldots, x_d^{p}) = \lim_{p \to \infty} \ehk (R_p)$ by \cite[Corollary~4.9]{Smirnov}, the Ehrhart-theoretic methods should be able to recover the formula for the limit Hilbert--Kunz multiplicity  \cite[Theorem~5.10]{CastilloRey2} as a weighted sum of volumes.
\end{remark}

\bibliographystyle{alpha}
\bibliography{refs}

\newcommand{\etalchar}[1]{$^{#1}$}
\begin{thebibliography}{{And}81}

\bibitem[AE08]{AberbachEnescu1}
Ian~M. Aberbach and Florian Enescu.
\newblock Lower bounds for {H}ilbert-{K}unz multiplicities in local rings of
  fixed dimension.
\newblock volume~57, pages 1--16. 2008.
\newblock Special volume in honor of Melvin Hochster.

\bibitem[AE13]{AberbachEnescu2}
Ian~M. Aberbach and Florian Enescu.
\newblock New estimates of {H}ilbert-{K}unz multiplicities for local rings of
  fixed dimension.
\newblock {\em Nagoya Math. J.}, 212:59--85, 2013.

\bibitem[{And}1881]{Andre1881}
{Andr{\'e}, D.}
\newblock Sur les permutations altern{\'e}es.
\newblock {\em Resal J. (3)}, 7:167--184, 1881.

\bibitem[BE04]{BlickleEnescu}
Manuel Blickle and Florian Enescu.
\newblock On rings with small {H}ilbert-{K}unz multiplicity.
\newblock {\em Proc. Amer. Math. Soc.}, 132(9):2505--2509, 2004.

\bibitem[BLM12]{BrennerLiMiller}
Holger Brenner, Jinjia Li, and Claudia Miller.
\newblock A direct limit for limit {H}ilbert-{K}unz multiplicity for smooth
  projective curves.
\newblock {\em J. Algebra}, 372:488--504, 2012.

\bibitem[Bre]{Brenner}
Holger Brenner.
\newblock Irrational {H}ilbert-{K}unz multiplicities.
\newblock Preprint, available at http://arxiv.org/abs/1305.5873.

\bibitem[CDHZ12]{CDHZ}
Olgur Celikbas, Hailong Dao, Craig Huneke, and Yi~Zhang.
\newblock Bounds on the {H}ilbert-{K}unz multiplicity.
\newblock {\em Nagoya Math. J.}, 205:149--165, 2012.

\bibitem[CRa]{CastilloRey2}
Joel Castillo-Rey.
\newblock Limit {F}-invariants of the simple singularities and the {F}ermat
  hypersurfaces.
\newblock Preprint, available at https://arxiv.org/abs/2609.39133.

\bibitem[CRb]{CastilloRey}
Joel Castillo-Rey.
\newblock Strong {W}atanabe-{Y}oshida conjecture for complete intersections.
\newblock Preprint, available at https://arxiv.org/abs/2506.09019.

\bibitem[CS23]{CoonsSullivant}
Jane~Ivy Coons and Seth Sullivant.
\newblock The {$h^*$}-polynomial of the order polytope of the zig-zag poset.
\newblock {\em Electron. J. Combin.}, 30(2):Paper No. 2.44, 20, 2023.

\bibitem[CSA24]{CoxSteibAberbach}
Nicholas~O. Cox-Steib and Ian~M. Aberbach.
\newblock Bounds for the {H}ilbert-{K}unz multiplicity of singular rings.
\newblock {\em Acta Math. Vietnam.}, 49(1):39--60, 2024.

\bibitem[CSR26]{CoxSteibRedner}
Nick Cox-Steib and Richard Redner.
\newblock The linear algebra of the {H}an-{M}onsky representation ring.
\newblock {\em Electron. J. Linear Algebra}, 42:329--348, 2026.

\bibitem[CSTZ]{CaminataShidelerTuckerZerman}
Alessio Caminata, Samuel Shideler, Kevin Tucker, and Francesco Zerman.
\newblock F-signature functions of diagonal hypersurfaces.
\newblock Preprint, available at https://arxiv.org/abs/2403.12863.

\bibitem[ES05]{EnescuShimomoto}
Florian Enescu and Kazuma Shimomoto.
\newblock On the upper semi-continuity of the {H}ilbert-{K}unz multiplicity.
\newblock {\em J. Algebra}, 285(1):222--237, 2005.

\bibitem[FMP26]{FMP}
Luis Ferroni, Alejandro~H. Morales, and Greta Panova.
\newblock Skew shapes, {E}hrhart positivity, and beyond.
\newblock {\em Proc. Lond. Math. Soc. (3)}, 133(2):Paper No. e70204, 31, 2026.

\bibitem[GM]{GesselMonsky}
Ira~M. Gessel and Paul Monsky.
\newblock The limit as $p \to \infty$ of the {H}ilbert-{K}unz multiplicity of
  $\sum x_i^{d_i}$.
\newblock Preprint, available at https://arxiv.org/abs/1007.2004.

\bibitem[Han92]{Han}
Chungsim Han.
\newblock {\em The {H}ilbert-{K}unz function of a diagonal hypersurface}.
\newblock ProQuest LLC, Ann Arbor, MI, 1992.
\newblock Thesis (Ph.D.)--Brandeis University.

\bibitem[HLP52]{HLP}
G.~H. Hardy, J.~E. Littlewood, and G.~P\'olya.
\newblock {\em Inequalities}.
\newblock Cambridge, at the University Press,, 1952.
\newblock 2d ed.

\bibitem[HM93]{HanMonsky}
C.~Han and P.~Monsky.
\newblock Some surprising {H}ilbert-{K}unz functions.
\newblock {\em Math. Z.}, 214(1):119--135, 1993.

\bibitem[JNS{\etalchar{+}}23]{JNSWY}
Jack Jeffries, Yusuke Nakajima, Ilya Smirnov, Keiichi Watanabe, and Ken-ichi
  Yoshida.
\newblock Lower bounds on {H}ilbert-{K}unz multiplicities and maximal
  {$F$}-signatures.
\newblock {\em Math. Proc. Cambridge Philos. Soc.}, 174(2):247--271, 2023.

\bibitem[Kah]{Kahane}
Yakob Kahane.
\newblock A proof of a conjecture of {Watanabe}--{Yoshida} via {Ehrhart}
  theory.
\newblock Preprint, available at https://arxiv.org/abs/2512.20442.

\bibitem[Kre76]{Kreweras}
Germain Kreweras.
\newblock Les pr\'{e}ordres totaux compatibles avec un ordre partiel.
\newblock {\em Math. Sci. Humaines}, (53):5--30, 1976.

\bibitem[Kun69]{Kunz1}
Ernst Kunz.
\newblock Characterizations of regular local rings of characteristic {$p$}.
\newblock {\em Amer. J. Math.}, 91:772--784, 1969.

\bibitem[Kun76]{Kunz2}
Ernst Kunz.
\newblock On {N}oetherian rings of characteristic {$p$}.
\newblock {\em Amer. J. Math.}, 98(4):999--1013, 1976.

\bibitem[Mena]{MengMonotonicity}
Cheng Meng.
\newblock {H}ilbert-{K}unz multiplicity of quadrics decreases.
\newblock Preprint, available at https://arxiv.org/abs/2606.04346.

\bibitem[Menb]{Meng}
Cheng Meng.
\newblock Limits of {F}-invariants and {R}iemann-{S}tieltjes integral.
\newblock Preprint, available at http://arxiv.org/abs/2507.13898.

\bibitem[Mon83]{Monsky}
Paul Monsky.
\newblock The {H}ilbert-{K}unz function.
\newblock {\em Math. Ann.}, 263(1):43--49, 1983.

\bibitem[Nag62]{Nagata}
Masayoshi Nagata.
\newblock {\em Local rings}, volume No. 13 of {\em Interscience Tracts in Pure
  and Applied Mathematics}.
\newblock Interscience Publishers (a division of John Wiley \& Sons, Inc.), New
  York-London, 1962.

\bibitem[OT20]{OhsugiTsuchiya1}
Hidefumi Ohsugi and Akiyoshi Tsuchiya.
\newblock Enriched chain polytopes.
\newblock {\em Israel J. Math.}, 237(1):485--500, 2020.

\bibitem[OT24]{OkadaTsuchiya}
Soichi Okada and Akiyoshi Tsuchiya.
\newblock Two enriched poset polytopes.
\newblock {\em Ann. Comb.}, 28(1):257--282, 2024.

\bibitem[Pet07]{Petersen}
T.~Kyle Petersen.
\newblock Enriched {$P$}-partitions and peak algebras.
\newblock {\em Adv. Math.}, 209(2):561--610, 2007.

\bibitem[PZ25]{PetersenZhuang}
T.~Kyle Petersen and Yan Zhuang.
\newblock Zig-zag {E}ulerian polynomials.
\newblock {\em European J. Combin.}, 124:Paper No. 104073, 30, 2025.

\bibitem[Slo]{oeis}
Neil Sloane.
\newblock The on-line encyclopedia of integer sequences.
\newblock Published electronically at http://oeis.org.

\bibitem[Smi20]{Smirnov}
Ilya Smirnov.
\newblock On semicontinuity of multiplicities in families.
\newblock {\em Doc. Math.}, 25:381--400, 2020.

\bibitem[Sta86]{Stanley}
Richard~P. Stanley.
\newblock Two poset polytopes.
\newblock {\em Discrete Comput. Geom.}, 1(1):9--23, 1986.

\bibitem[Sta10]{StanleySurvey}
Richard~P. Stanley.
\newblock A survey of alternating permutations.
\newblock In {\em Combinatorics and graphs}, volume 531 of {\em Contemp.
  Math.}, pages 165--196. Amer. Math. Soc., Providence, RI, 2010.

\bibitem[Tri18]{TrivediDensity}
V.~Trivedi.
\newblock Hilbert-{K}unz density function and {H}ilbert-{K}unz multiplicity.
\newblock {\em Trans. Amer. Math. Soc.}, 370(12):8403--8428, 2018.

\bibitem[Tri23]{Trivedi}
Vijaylaxmi Trivedi.
\newblock The {H}ilbert-{K}unz density functions of quadric hypersurfaces.
\newblock {\em Adv. Math.}, 430:Paper No. 109207, 63, 2023.

\bibitem[WY00]{WatanabeYoshida}
Keiichi Watanabe and Ken-ichi Yoshida.
\newblock Hilbert-{K}unz multiplicity and an inequality between multiplicity
  and colength.
\newblock {\em J. Algebra}, 230(1):295--317, 2000.

\bibitem[WY05]{WatanabeYoshida3d}
Keiichi Watanabe and Ken-ichi Yoshida.
\newblock Hilbert-{K}unz multiplicity of three-dimensional local rings.
\newblock {\em Nagoya Math. J.}, 177:47--75, 2005.

\bibitem[Yos19]{Yoshida}
{Ken-ichi} Yoshida.
\newblock Small {H}ilbert-{K}unz multiplicity and {($A_1$)}-type singularity.
\newblock Proceedings of the 4th Japan-Vietnam Joint Seminar on Commutative
  Algebra by and for Young Mathematicians, Meiji University, 2019.

\end{thebibliography}

\appendix

\section{Monotonicity in the characteristic via the spectrum of the Han--Monsky matrices}\label{AI}
\label{app:monotonicity}

Throughout the appendix, $p\geq 3$ is a prime and $a=(p-1)/2$.
We now give a different proof of the principal result of \cite{MengMonotonicity}: for every fixed
$d\geq 4$ the function $p \mapsto \ehk(A_{p,d})$ is a strictly decreasing function for $p\geq 3$, as conjectured in~\cite{Yoshida}. 
In fact, the rational
function of Corollary~\ref{cor: rational function}, with its obvious common
factors removed, is strictly decreasing on the whole half-line $p>0$ and takes
the value $1+1/d$ at $p=0$ (Proposition~\ref{app:prop monotone}). As a
by-product we obtain the bounds $1+\En{d}/d!<\ehk(A_{p,d})<1+1/d$ and the
second-order term of the expansion of $\ehk(A_{p,d})$ at $p=\infty$.

\subsection{The Han--Monsky functional as a trigonometric power sum}
\label{app:sub spectral}

As a principal ingredient we will use the following result of Cox-Steib and Redner in~\cite{CoxSteibRedner} (note that \cite{CoxSteibRedner}  denotes $\Mat(\lambda_i)$ by  $A_i$ and $\Mat (\delta_l) =\sum_{i<l}(-1)^i \Mat(\lambda_i)$ by $B_l$). 

\begin{theorem}[{\cite[Theorems~2.2, 2.3, 2.4, and 3.1]{CoxSteibRedner}}]\label{app:thm CSR}
For $n=1,\ldots,p$ let $v_n\in\RR^p$ be the vector with coordinates
$(v_n)_r=\sin\frac{(2r-1)n\pi}{2p}$, $r=1,\ldots,p$.
\begin{enumerate}
\item[(a)] 
$\Mat(\lambda_i)v_n=\eta_{i,n}v_n$, where
$\displaystyle\eta_{i,n}=\frac{\sin\frac{(2i+1)n\pi}{2p}}{\sin\frac{n\pi}{2p}}$.
\item[(b)]
$\Mat(\delta_l)v_n=\gamma_{l,n}v_n$, where
$\displaystyle\gamma_{l,n}=(-1)^{l+1}\frac{\sin\frac{ln\pi}{p}}{\sin\frac{n\pi}{p}}$
for $1\leq n\leq p-1$, and $\gamma_{l,p}=l$.
\item[(c)] $v_1,\ldots,v_p$ are pairwise orthogonal vectors, $\|v_n\|^2=p/2$ for
$1\leq n\leq p-1$, and $\|v_p\|^2=p$.
\end{enumerate}
\end{theorem}

Orthogonality in~(c) holds because $\Mat(\lambda_1)$ is symmetric and its
eigenvalues $\eta_{1,n}=1+2\cos\frac{n\pi}{p}$ are pairwise distinct.
Since $b\mapsto \Mat(b)$ is a ring homomorphism, (a) shows that $v_n$ is an
eigenvector of $\Mat(b)$ for every $b\in\Gamma_{k,p}$; we write $\chi_n(b)$ for
the eigenvalue. Thus $\chi_n\colon\Gamma_{k,p}\to\RR$ is the ring homomorphism
with $\chi_n(\lambda_i)=\eta_{i,n}$, and $\chi_n(\delta_l)=\gamma_{l,n}$ by~(b).

\begin{corollary}\label{app:cor weights}
Put $\omega_n=\frac{2}{p}\sin^2\frac{n\pi}{2p}$ for $1\leq n\leq p-1$ and
$\omega_p=\frac1p$. Then for every $b\in\Gamma_{k,p}$,
\[
D_k(b)=[\Mat(b)]_{1,1}=\sum_{n=1}^{p}\omega_n\,\chi_n(b).
\]
\end{corollary}
\begin{proof}
By Subsection~\ref{ss: HM linear algebra}, 
$D_k(b)$ is the first entry of the first column of $\Mat(b)$, which is the coordinate vector of $b\cdot\lambda_0=b$. By Theorem~\ref{app:thm CSR}(c),
$I=\sum_{n=1}^pv_nv_n^T/\|v_n\|^2$, hence
$\Mat(b)=\sum_n\chi_n(b)\,v_nv_n^T/\|v_n\|^2$. Take the $(1,1)$-entry and use
$(v_n)_1=\sin\frac{n\pi}{2p}$, which equals $1$ for $n=p$.
\end{proof}

\begin{proposition}\label{app:prop trig}
Put $x_j=(-1)^{a+j}\cos\frac{j\pi}{p}$ for $-a\leq j\leq a$. Then $x_{-j}=x_j$,
$x_0=(-1)^a$, $|x_j|\geq\cos\frac{a\pi}{p}=\sin\frac{\pi}{2p}>0$, and for every
$k\geq 1$
\[
D_k(\lambda_a^{k})=\frac1p\sum_{j=-a}^{a}x_j^{\,2-k},\qquad
D_k\bigl((\delta_a+\delta_{a+1})^{k}\bigr)=p^{k-1}+\frac1p\sum_{j=-a}^{a}(1-x_j^2)\,x_j^{-k}.
\]
\end{proposition}
\begin{proof}
By Corollary~\ref{app:cor weights}, $D_k(\lambda_a^k)=\sum_n\omega_n\eta_{a,n}^k$
and $D_k((\delta_a+\delta_{a+1})^k)=\sum_n\omega_n (\gamma_{a,n}+\gamma_{a+1,n})^k$. 

Since $2a+1=p$, $\eta_{a,n}=\sin\frac{n\pi}{2}\big/\sin\frac{n\pi}{2p}$, which
vanishes for even $n<p$. For odd $n=p-2j<p$ ($1\leq j\leq a$) we have
$\sin\frac{n\pi}{2p}=\cos\frac{j\pi}{p}$ and
$\sin\frac{n\pi}{2}=(-1)^j\sin\frac{p\pi}{2}=(-1)^{a+j}$, so
$\eta_{a,p-2j}=1/x_j$ and $\omega_{p-2j}=\frac2px_j^2$; finally
$\eta_{a,p}=(-1)^a=1/x_0$ and $\omega_p=\frac1px_0^2$. Using $x_{-j}=x_j$,
\[
D_k(\lambda_a^k)=\frac1p\Bigl(x_0^{2-k}+2\sum_{j=1}^{a}x_j^{2-k}\Bigr)
=\frac1p\sum_{j=-a}^{a}x_j^{2-k}.
\]

Second, Theorem~\ref{app:thm CSR}(b) gives for $n < p$
\[
\gamma_{a,n}+\gamma_{a+1,n} =(-1)^{a+1}\,\frac{\sin(\frac{n\pi}{2}-\frac{n\pi}{2p})-\sin(\frac{n\pi}{2}+\frac{n\pi}{2p})}{\sin 2\varphi}
=(-1)^a\,\frac{\cos\frac{n\pi}{2}}{\cos\frac{n\pi}{2p}},
\]
which vanishes for odd $n$, while for $n=2j$ ($1\leq j\leq a$) it equals
$(-1)^{a+j}/\cos\frac{j\pi}{p}=1/x_j$, with
$\omega_{2j}=\frac2p\sin^2\frac{j\pi}{p}=\frac2p(1-x_j^2)$. Finally
$\gamma_{a,p}+\gamma_{a+1,p} =a+(a+1)=p$ and $\omega_p=\frac1p$. Hence
$D_k((\delta_a+\delta_{a+1})^k)=p^{k-1}+\frac2p\sum_{j=1}^a(1-x_j^2)x_j^{-k}$,
which is the claim because the term $j=0$ vanishes.
\end{proof}

\begin{corollary}[cf. {\cite[Section~4.1]{CoxSteibRedner}}]\label{app:cor difference}
For $d\geq 0$ put
\begin{equation}\label{app:eq beta}
\beta_d(p) :=\frac1p\sum_{j=-a}^{a}x_j^{-(d+1)} = \frac 1p\sum_{j=-a}^{a} (-1)^{(a+j)(d+1)} \cos^{-(d+1)} \frac{j\pi}{p}.
\end{equation}
Then for every $d\geq 2$,
\[
\ehk(A_{p,d}) =1+\frac{\beta_d(p)-\beta_{d-2}(p)}{p^d-\beta_{d-2}(p)}.
\]
\end{corollary}
\begin{proof}
Since $\lambda_a=(-1)^a(\delta_{a+1}-\delta_a)$, the denominator in
Corollary~\ref{cor HanMonsky} is $p^d-D_k(\lambda_a^{d+1})$. By
Proposition~\ref{app:prop trig} with $k=d+1$, we get
$D_k(\lambda_a^{d+1})=\beta_{d-2}(p)$ and
$D_k((\delta_a+\delta_{a+1})^{d+1})=p^d+\beta_d(p)-\beta_{d-2}(p)$.
\end{proof}

\subsection{The generating function and the coefficient-ratio theorem}
\label{app:sub product}
Motivated by Corollary~\ref{app:cor difference} we now proceed to study $\beta_d(p)$. 
\begin{theorem}\label{app:thm generating}
For $|z|<1$ put $u=u(z)=\arcsin z$ and $c=c(z)=\sqrt{1-z^2}$. Then 
for $|z|<\sin\frac{\pi}{2p}$
\[
\sum_{d\geq 0}\beta_d(p)\,z^d=\frac{\sec pu + \tan pu}{c}.
\]
\end{theorem}
\begin{proof}
By \eqref{app:eq beta} and $|x_j|\geq\sin\frac{\pi}{2p}>|z|$,
\[
\sum_{d\geq 0}\beta_d(p)\,z^d=\frac1p\sum_{j=-a}^{a}\sum_{d\geq0}x_j^{-(d+1)}z^d
=\frac1p\sum_{j=-a}^{a}\frac{1}{x_j-z}.
\]
Let $\mathcal T_p$ be the Chebyshev polynomial of the first kind,
$\mathcal T_p(\cos\theta)=\cos p\theta$; it has degree $p$ and leading
coefficient $2^{p-1}$. Put $\Phi(z)=\mathcal T_p(z)-(-1)^a$. We have
$x_j=\cos\theta_j$ with $\theta_j=\frac{j\pi}{p}+(a+j)\pi$, and
$\cos p\theta_j=\cos\bigl(j\pi+p(a+j)\pi\bigr)=(-1)^a$ because $p$ is odd, so
$\Phi(x_j)=0$. For $1\leq |j|\leq a$ we have $\sin\theta_j\neq 0$ and
$\sin p\theta_j=0$, so differentiating $\mathcal T_p(\cos\theta)=\cos p\theta$
at $\theta_j$ gives $\mathcal T_p'(x_j)=0$: $x_j$ is at least a double root of
$\Phi$. The numbers $x_0,\ldots,x_a$ are pairwise distinct, since
$|x_j|=\cos\frac{j\pi}{p}$ is strictly decreasing; counting multiplicities,
$1+2a=p=\deg\Phi$, so $\Phi(z)=2^{p-1}\prod_{j=-a}^{a}(z-x_j)$ and
\[
\frac1p\sum_{j=-a}^{a}\frac1{x_j-z}=-\frac{\Phi'(z)}{p\,\Phi(z)}.
\]
For real $z$ with $|z|<\sin\frac{\pi}{2p}$ we have $|pu|<\frac\pi2$ and
$\arccos z=\frac\pi2-u$, so
$\mathcal T_p(z)=\cos(a\pi+\frac\pi2-pu)=(-1)^a\sin pu$ and, differentiating in
$u$, $\mathcal T_p'(z)\,c=(-1)^ap\cos pu$. Hence
\[
-\frac{\Phi'(z)}{p\,\Phi(z)}
=\frac{\cos pu}{c\,(1-\sin pu)}
=\frac{1+\sin pu}{c\,\cos pu}=\frac{\sec pu + \tan pu}{c}.
\]
Both sides are analytic for $|z|<\sin\frac{\pi}{2p}$ (the right-hand side
because $|\arcsin z|\leq\arcsin|z|$), so the identity holds on that disc.
\end{proof}

Throughout the rest of this subsection, for $d\geq 0$ we write $d=2M+\varepsilon$
with $\varepsilon\in\{0,1\}$ and put
\[
X_d = \bigl\{(2r-1+\varepsilon)^2 \mid r \in [M] \bigr \} 
\; (\text{i.e., the odd squares if $d$ is even, the even squares if $d$ is odd}).
\]
For a finite set $X$ of positive numbers, $e_r(X)$ denotes its $r$-th
elementary symmetric function, with $e_0(X)=1$ and $e_r(X)=0$ unless
$0\leq r\leq |X|$.

\begin{theorem}[Explicit formula]\label{app:thm product}
For $d\geq 0$ define
\[
B_d(p) := \frac{p^\varepsilon}{d!}\sum_{j=0}^{M}\En{2j+\varepsilon}\,e_{M-j}(X_d)\,p^{2j}
\in\mathbb Q[p].
\]
Then $\beta_d(p)=B_d(p)$ for every prime $p\geq 3$. The polynomial $B_d$ has
the parity of $d$, every coefficient allowed by parity is positive, its
leading coefficient is $\En{d}/d!$, and $B_d(1)=1$.
\end{theorem}
\begin{proof}
Put $u(z)=\arcsin z$
and 
$G_k(z)=u(z)^k/k!$ for $k\geq 0$; since $u=z+O(z^3)$, $G_k=O(z^k)$.
For brevity, let $[z^d]f$ to denote the coefficient of $z^d$ in a power series $f(z)$. 
Expanding the right-hand side in Theorem~\ref{app:thm generating}
gives
\begin{equation}\label{app:eq beta expansion}
\beta_d(p)=\sum_{k=0}^{d}\En{k}\,p^k\,[z^d]\frac{G_k}{c},
\end{equation}
a polynomial in $p$, and it suffices to show that it equals $B_d(p)$.
Since $u'=1/c$, we have $G_{k+1}'=G_k/c$, hence
$[z^d](G_k/c)=(d+1)[z^{d+1}]G_{k+1}$.  To compute these numbers, consider
$f=e^{yu}=\sum_kG_ky^k=\sum_{n\geq0}R_n(y)\,z^n/n!$ with
$R_n(y)=\sum_k n!\,[z^n]G_k y^k$. From $f'=yf/c$ and $(1/c)'=z/c^3$ we get
$(1-z^2)f''=y^2f+zf'$, and comparing the coefficients of $z^n/n!$,
\[
R_{n+2}(y)=(y^2+n^2)R_n(y),\qquad R_0=1,\quad R_1=y,
\]
so $R_{d+1}(y)=y^{\varepsilon+1}\prod_{r=1}^{M}\bigl(y^2+(2r-1+\varepsilon)^2\bigr)$.
Consequently the coefficient of $p^k$ in \eqref{app:eq beta expansion} is
\[
\frac{\En{k}}{d!}(d+1)!\,[z^{d+1}]G_{k+1} =\frac{\En{k}}{d!}[y^{k-\varepsilon}]\prod_{x\in X_d}(y^2+x) = 
\begin{cases}
\frac{\En{k}}{d!}e_{M-j}(X_d) &  \text{if $k=2j+\varepsilon$}\\
0 & \text{otherwise.}
\end{cases}
\]
This proves $\beta_d=B_d$ and the assertions on parity, positivity
and the leading coefficient. The last claim is straightforward from \eqref{app:eq beta expansion}. 
\end{proof}

We will need the following elementary lemma. 

\begin{lemma}\label{app:lem cumulative}
Let $m\geq 0$, let $b_0,\ldots,b_m$ be positive and $a_0,\ldots,a_m$ real
numbers with $a_0/b_0<a_1/b_1<\cdots<a_m/b_m$, and put
$\bar a_j=\sum_{i\leq j}a_i$, $\bar b_j=\sum_{i\leq j}b_i$. Then
$\bar a_j/\bar b_j$ is strictly increasing in $j$.
\end{lemma}
\begin{proof}
$\bar a_j/\bar b_j=\sum_{i\leq j}\frac{b_i}{\bar b_j}\cdot\frac{a_i}{b_i}$ is a
weighted average of $a_0/b_0,\ldots,a_j/b_j$ with positive weights, so
$\bar a_j/\bar b_j<a_{j+1}/b_{j+1}$, and therefore
\[
\frac{\bar a_{j+1}}{\bar b_{j+1}}-\frac{\bar a_j}{\bar b_j}
=\frac{b_{j+1}}{\bar b_{j+1}}
\Bigl(\frac{a_{j+1}}{b_{j+1}}-\frac{\bar a_j}{\bar b_j}\Bigr)>0
\qquad(0\leq j<m).\qedhere
\]
\end{proof}

\begin{theorem}[Coefficient-ratio theorem]\label{app:thm ratios}
Let $d\geq 2$, so that $M\geq 1$. There are polynomials
$\mathcal N_d,\mathcal D_d\in\mathbb Q[t]$ of degree $M-1$ such that
\begin{equation}\label{app:eq N and D}
B_d(p)-B_{d-2}(p)=p^\varepsilon(p^2-1)\,\mathcal N_d(p^2) \text{ and }
p^d-B_{d-2}(p)=p^\varepsilon(p^2-1)\,\mathcal D_d(p^2).
\end{equation}
Writing $\mathcal N_d(t)=\sum_{j=0}^{M-1}\alpha_jt^j$ and
$\mathcal D_d(t)=\sum_{j=0}^{M-1}\beta_jt^j$, one has $\alpha_j>0$, $\beta_j>0$
and
\begin{equation}\label{app:eq ratio chain}
\frac1d=\frac{\alpha_0}{\beta_0}>\frac{\alpha_1}{\beta_1}>\cdots>
\frac{\alpha_{M-1}}{\beta_{M-1}}=\frac{\En{d}}{d!}.
\end{equation}
For $d=2,3$ (that is, $M=1$) the chain collapses to the single equality
$1/d=\alpha_0/\beta_0=\En{d}/d!$.
\end{theorem}
\begin{proof}
By Theorem~\ref{app:thm product}, applied to $d$ and to $d-2=2(M-1)+\varepsilon$,
\begin{equation}\label{app:eq ab}
B_d(p)=p^\varepsilon\sum_{i=0}^{M}a_ip^{2i},\quad
B_{d-2}(p)=p^\varepsilon\sum_{i=0}^{M-1}b_ip^{2i},\quad
a_i=\frac{\En{2i+\varepsilon}\,e_{M-i}(X_d)}{d!},\quad
b_i=\frac{\En{2i+\varepsilon}\,e_{M-1-i}(X_{d-2})}{(d-2)!},
\end{equation}
where all $a_i$ and $b_i$ are positive, $a_M=\En{d}/d!$, and
$\sum_ia_i=\sum_ib_i=1$ because $B_d(1)=B_{d-2}(1)=1$.
Hence the polynomials $\sum_{i=0}^{M}(a_i-b_i)t^i$ (with $b_M=0$) and
$t^M-\sum_{i=0}^{M-1}b_it^i$ vanish at $t=1$; let $\mathcal N_d$ and $\mathcal D_d$ be their
quotients by $t-1$. Then 
\eqref{app:eq N and D} follows from \eqref{app:eq ab} and
$p^d=p^\varepsilon(p^2)^M$.

Since 
$(t-1)\mathcal N_d(t)=\sum_i(a_i-b_i)t^i$ and
$(t-1)\mathcal D_d(t)=t^M-\sum_ib_it^i$, we must have, for $0\leq j\leq M-1$,
\begin{equation}\label{app:eq cumulative}
\alpha_j=\sum_{i\leq j} (b_i - a_i),\qquad \beta_j=\sum_{i\leq j}b_i.
\end{equation}
In particular $\beta_j>0$, $\beta_{M-1} =1$ and
$\alpha_{M-1}=1-(1-a_M)=a_M=\En{d}/d!$, so both polynomials have degree
exactly~$M-1$. 

Let $X=X_{d-2}$, a list of $M-1$ positive numbers (empty if $M=1$). Since
$X_d=X\cup\{(d-1)^2\}$, we have
$e_{M-i}(X_d)=e_{M-i}(X)+(d-1)^2e_{M-1-i}(X)$, where $e_{M}(X)=0$.
By \eqref{app:eq ab}, for
$0\leq i\leq M-1$,
\begin{equation}\label{app:eq local ratios}
\frac{a_i}{b_i}=\frac{1}{d(d-1)}\Bigl((d-1)^2+\frac{e_{M-i}(X)}{e_{M-1-i}(X)}\Bigr).
\end{equation}
By Newton's inequalities \cite[Theorem~53]{HLP}, 
$e_r(X)/e_{r-1}(X)$ is strictly decreasing for $1\leq r\leq M-1$, and since
$e_{M}(X)/e_{M-1}(X)=0<e_1(X)/e_0(X)$, the sequence $a_i/b_i$ in
\eqref{app:eq local ratios} is strictly increasing in $i$ (for $M\leq 2$ there
is nothing to prove). By Lemma~\ref{app:lem cumulative} and \eqref{app:eq cumulative},
$\alpha_j/\beta_j=1-(\sum_{i\leq j} a_i)/(\sum_{i\leq j}b_i)$ 
is strictly decreasing in $j$. Its last
value is $\alpha_{M-1}/\beta_{M-1}=\En{d}/d!>0$; hence every $\alpha_j/\beta_j$ is
positive and so is every $\alpha_j$. Finally, \eqref{app:eq local ratios} with
$i=0$ gives 
\[
\alpha_0/\beta_0=1-a_0/b_0=1-\frac{(d-1)^2}{d(d-1)}=\frac1d. \qedhere
\]
\end{proof}

\subsection{Monotonicity in the characteristic}\label{app:sub monotone}

\begin{proposition}
\label{app:prop monotone}
Fix $d\geq 2$, let $\mathcal N_d,\mathcal D_d$ be as in
Theorem~\ref{app:thm ratios}, and put
\[
H_d(p)=1+\frac{\mathcal N_d(p^2)}{\mathcal D_d(p^2)}.
\]
Then $H_d$ is an even rational function, regular on all of $\RR$, and
\begin{enumerate}
\item[(a)] $H_d(0)=1+\frac1d$ and $\lim_{p\to\infty}H_d(p)=1+\frac{\En{d}}{d!}$;
\item[(b)] $H_d$ is the constant $1+\frac1d$ for $d=2,3$, and $H_d'(p)<0$ for
all $p>0$ if $d\geq 4$;
\item[(c)] as $p\to\infty$,
\[
H_d(p)=1+\frac{\En{d}}{d!}+\frac{\En{d-2}}{(d-2)!}\Bigl(\frac{\En{d}}{d!}+\frac{d-5}{6}\Bigr)\,p^{-2}+O(p^{-4}).
\]
\end{enumerate}
\end{proposition}
\begin{proof}
Since every $\beta_j>0$, we have $\mathcal D_d(t)>0$ for $t\geq 0$, so $H_d$
is regular on $\RR$.

(a) By Theorem~\ref{app:thm ratios}, $H_d(0)=1+\alpha_0/\beta_0=1+1/d$, and,
since $\deg\mathcal N_d=\deg\mathcal D_d=M-1$, with $d=2M+\varepsilon$ as in
Subsection~\ref{app:sub product},
$\lim_{p\to\infty}H_d(p)=1+\alpha_{M-1}/\beta_{M-1}=1+\En{d}/d!$.

(b) If $d=2,3$, then $M=1$ and $\mathcal N_d/\mathcal D_d=\alpha_0/\beta_0=1/d$ is
constant. Let $d\geq 4$, so $M\geq 2$, and write $H_d(p)-1=f(p)/g(p)$ with
$f(p)=\sum_{j=0}^{M-1}\alpha_jp^{2j+1}$ and $g(p)=\sum_{j=0}^{M-1}\beta_jp^{2j+1}$.
Since $\beta_j>0$ and the ratios $\alpha_j/\beta_j$ are strictly decreasing by
Theorem~\ref{app:thm ratios}, we obtain that  
\[
g(p)^2\Bigl(\frac{f}{g}\Bigr)'(p)=g(p)f'(p)-g'(p)f(p)
=-2\sum_{0\leq i<j\leq M-1}(j-i)\bigl(\alpha_i\beta_j-\alpha_j\beta_i\bigr)\,p^{2(i+j)+1} < 0.
\]
Hence
$H_d'(p)<0$ for every $p>0$.

(c) By Theorem~\ref{app:thm ratios},
$H_d(p)-1=\bigl(B_d(p)-B_{d-2}(p)\bigr)/\bigl(p^d-B_{d-2}(p)\bigr)$ as rational
functions. In Theorem~\ref{app:thm product} the two highest terms of $B_d$ come
from $j=M,M-1$, with $e_0(X_d)=1$ and
$e_1(X_d)=\sum_{r=1}^{M}(2r-1+\varepsilon)^2=\frac{d(d^2-1)}{6}$ for both
parities, so
$B_d(p)=\frac{\En{d}}{d!}p^d+\frac{d+1}{6}\frac{\En{d-2}}{(d-2)!}p^{d-2}+O(p^{d-4})$.
Hence
\[
H_d(p)-1=\frac{\frac{\En{d}}{d!}p^d+\frac{d-5}{6}\frac{\En{d-2}}{(d-2)!}p^{d-2}+O(p^{d-4})}
{p^d-\frac{\En{d-2}}{(d-2)!}p^{d-2}+O(p^{d-4})}
=\frac{\En{d}}{d!}+\frac{\En{d-2}}{(d-2)!}\Bigl(\frac{\En{d}}{d!}+\frac{d-5}{6}\Bigr)\,p^{-2}+O(p^{-4}).
\qedhere
\]
\end{proof}

\begin{corollary}[Hilbert--Kunz multiplicity of $A_{p,d}$]\label{app:cor bounds}
For every prime $p\geq 3$ and every $d\geq 2$,
\[
\ehk(A_{p,d})=H_d(p)=1+\frac{\mathcal N_d(p^2)}{\mathcal D_d(p^2)},
\]
where $\mathcal N_d$ and $\mathcal D_d$ have degree exactly
$\lfloor d/2\rfloor-1$ and positive coefficients; in particular, the
polynomials in Corollary~\ref{cor: rational function} can be taken with
positive coefficients. Consequently, 
for $d \geq 4$
\begin{enumerate}
\item[(a)] the function $p\mapsto\ehk(A_{p,d})$ is strictly
decreasing on the primes $p\geq 3$;
\item[(b)] $\;1+\En{d}/d!<\ehk(A_{p,d})<1+1/d$;
\item[(c)] $\displaystyle \ehk(A_{p,d})=1+\frac{\En{d}}{d!}+\frac{\En{d-2}}{(d-2)!}\Bigl(\frac{\En{d}}{d!}+\frac{d-5}{6}\Bigr)\,p^{-2}+O(p^{-4})$; this makes
Corollary~\ref{cor convergence rate} explicit.
\end{enumerate}
\end{corollary}
\begin{proof}
For a prime $p\geq 3$, Corollary~\ref{app:cor difference} and
Theorem~\ref{app:thm product} give
$\ehk(A_{p,d})=1+\bigl(B_d(p)-B_{d-2}(p)\bigr)/\bigl(p^d-B_{d-2}(p)\bigr)$, and
Theorem~\ref{app:thm ratios} turns this into
$\ehk(A_{p,d})=H_d(p)$, because the common factor $p^\varepsilon(p^2-1)$ is
nonzero. The degrees and the positivity of the coefficients of $\mathcal N_d$
and $\mathcal D_d$ are part of Theorem~\ref{app:thm ratios}, and (a)--(c)
follow from Proposition~\ref{app:prop monotone}.
\end{proof}

\begin{example}\label{app:ex d=6}
For $d=4$, $\mathcal N_4(t)=(5t+3)/24$ and $\mathcal D_4(t)=(2t+1)/2$, with
ratios $\frac14>\frac5{24}$. For $d=6$,
$\mathcal N_6(t)=\frac{45+86t+61t^2}{720}$ and
$\mathcal D_6(t)=\frac{9+19t+24t^2}{24}$, with ratios
$\frac16>\frac{43}{285}>\frac{61}{720}=\frac{\En6}{6!}$.
\end{example}

\begin{remark}\label{app:rem kahane}
The proof of Corollary~\ref{cor main Ehrhart} shows that
$D_k(\lambda_a^{d+1})=[N_a^{d+1}]_{1,1}=\Ehr{E_{d-2}}(a)$, and
Corollary~\ref{cor matrices} with $n=2$ together with
Theorem~\ref{thm Fibonacci} gives
$D_k\bigl((\delta_a+\delta_{a+1})^{d+1}\bigr)=[T_a^{d+1}]_{1,1}=p^d+2^d \Ehr{F_d} (a-1)$.
Comparing with Proposition~\ref{app:prop trig} and
Theorem~\ref{app:thm product} yields the polynomial
identities
\[
B_d(p)=\Ehr{E_d}\Bigl(\frac{p-1}{2}\Bigr),\qquad
2^d\Ehr{F_d}\Bigl(\frac{p-3}{2}\Bigr)=B_d(p)-B_{d-2}(p)\quad(d\geq 2),
\]
so Theorem~\ref{app:thm product} is an explicit formula for the Ehrhart
polynomial of the extended Fibonacci polytope that can be compared to \cite[Theorem~1.10]{MengMonotonicity}.

Since $F_n$ is the order polytope of the zigzag poset $Z_n$,
$\Ehr{F_n}(k)=\Omega(Z_n;k+1)$, so the shifted order polynomial
$\Omega(Z_n;t-\tfrac12)$ of Kahane~\cite{Kahane} equals
$\Ehr{F_n}(t-\tfrac32)$, and $2^n\,\Omega(Z_n;t-\tfrac12)=B_n(2t)-B_{n-2}(2t)$
for $n\geq 2$. Writing $\mathcal S(u) := \sum_k\frac{\En{k}}{k!}u^k$,
Theorem~\ref{app:thm generating} takes the form of a Hadamard product\footnote{By definition, $\sum_ka_kp^k\odot_p\sum_kb_kp^k=\sum_ka_kb_kp^k$.} in $p$:
\[
(1-z^2)\sum_{d\geq 0}B_d(p)\,z^d=\mathcal S(p)\odot_p\frac{c}{1-pu},
\]
and after the substitutions $p=2t$, $z\mapsto z/2$ the coefficients of $z^n$,
$n\geq 2$, give
\[
\sum_{n\geq 0}\Omega\bigl(Z_n;t-\tfrac12\bigr)z^n
=(\tan t+\sec t)\odot_t\frac{\sqrt{1-z^2/4}}{1-2t\arcsin(z/2)}
\qquad(\text{coefficients of }z^n,\ n\neq 1),
\]
which is \cite[Theorem~5.2]{Kahane}.
Theorem~\ref{app:thm product} evaluates the weighted sums of
\cite[Theorem~4.17]{Kahane} in closed form: with $n=2M+\varepsilon$ as in
Subsection~\ref{app:sub product} and $k=2j+\varepsilon$,
\[
\sum_{\mu\in\mathrm{SDecompo}_k(Z_n)}w(\mu)
=\frac{k!}{n!\,2^{\,n-k}}\Bigl(e_{M-j}(X_{n-2})-(n-1)\,e_{M-1-j}(X_{n-2})\Bigr).
\]
For $k=n-2$ this is $f_{n,n-2}=n(n-1)(n-5)\En{n-2}/24$
\cite[Example~4.19]{Kahane}; in the notation of
Corollary~\ref{c Ehrhart expand} it reads
$\sum_m\binom m2 s_d(m)=\bigl(\En{d}(3d^2-17d+25)+(d-5)\En{d-2}\bigr)/24$, a
closed formula for the second diagonal of Kreweras' table in
Remark~\ref{remark Kreweras}. Finally, Theorem~\ref{app:thm ratios} shows that
$\Omega(Z_n;t-\tfrac12)\big/\bigl(t^{\varepsilon}(4t^2-1)\bigr)$ has positive
coefficients for every $n\geq 2$.
\end{remark}

\end{document}